\documentclass{amsart}
\usepackage[T1]{fontenc}
\usepackage{amssymb}
\usepackage{mathrsfs}

\theoremstyle{plain}
\newtheorem{prop}{Proposition}[section]
\newtheorem{thm}[prop]{Theorem}
\newtheorem{lem}[prop]{Lemma}
\newtheorem{cor}[prop]{Corollary}

\numberwithin{equation}{section}

\theoremstyle{definition}
\newtheorem{dfn}[prop]{Definition}

\theoremstyle{remark}
\newtheorem*{rmk}{Remark}

\DeclareMathOperator{\kernel}{Ker}
\DeclareMathOperator{\cokernel}{Coker}
\DeclareMathOperator{\ran}{Ran}
\DeclareMathOperator{\re}{Re}
\DeclareMathOperator{\im}{Im}
\DeclareMathOperator{\supp}{supp}

\newcommand{\N}{\mathbb{N}}
\newcommand{\R}{\mathbb{R}}
\newcommand{\C}{\mathbb{C}}

\newcommand{\Li}{\mathcal{L}}
\newcommand{\De}{\mathscr{D}\, '}
\newcommand{\E}{\mathscr{E}}
\newcommand{\I}{\mathrm{Id}}
\newcommand{\ol}{\overline}

\begin{document}
\title[Microlocal properties of the range of systems of principal type]
{On the microlocal properties of the range of systems of principal type}

\author{Nils Dencker}
\address{Centre for Mathematical Sciences, University of Lund, Box 118,
SE-221 00 Lund, Sweden}
\email{dencker@maths.lth.se}

\author{Jens Wittsten}
\address{Graduate school of Human and Environmental Studies\\
Kyoto University\\
Yoshida Nihonmatsu-cho, Sakyo-ku\\
Kyoto 606-8501\\
Japan}
\thanks{Research of second author supported by the Japan Society for the Promotion of Science.}
\email{jens.wittsten@math.lu.se}
\subjclass[2010]{Primary 35S05; Secondary 35A02, 58J40, 47G30}

\begin{abstract}
The purpose of this paper is to study microlocal conditions
for inclusion relations between the ranges of square systems
of pseudodifferential operators which fail to be locally solvable.
The work is an extension of
earlier results for the scalar case in this direction, where analogues of
results by L. H{\"o}rmander about inclusion relations between the ranges of
first order differential operators with coefficients in $C^\infty$
which fail to be locally solvable were obtained. We shall study the properties of the 
range of systems of principal type with constant characteristics
for which condition $(\varPsi)$ is known to be equivalent to microlocal solvability.
\end{abstract}

\maketitle

\section{Introduction}

\noindent
In this paper we shall study the properties of the range of a
square system of classical
pseudodifferential operators $P\in \varPsi_{\mathrm{cl}}^m(X)$
on a $C^\infty$
manifold $X$ of dimension $n$,
acting on distributions $\De(X,\mathbb{C}^N)$ with values in 
$\C^N$; if $u\in\De(X,\C^N)$ then $u=(u_j)_{j=1,\ldots,N}$
where $u_j\in\De(X)$. If $P=(P_{jk})$ is an $N\times N$ system, then
$Pu\in\De(X,\C^N)$ is defined by
\begin{equation}\label{defsys}
(Pu)_j=\sum_{k=1}^N P_{jk}u_k,\quad 1\le j \le N.
\end{equation}
Here classical means that the symbol of $P$ is an asymptotic sum
$P_m+P_{m-1}+\ldots$ of matrix valued smooth functions
where $P_j(x,\xi)$ is homogeneous of degree $j$ in $\xi$,
and $P_m$ is the principal symbol.

We shall restrict our study
to systems of principal type, which means that the principal symbol
vanishes of first order on the kernel, see Definition \ref{def:systemprincip}.
We shall also assume that all (systems of) operators are properly supported, that is, both projections
from the support of the operator kernel in $X\times X$ to $X$ are proper maps.
For such $N\times N$ systems,
local solvability at a compact set $M\subset X$ means that for every $f$ in a
subspace of $C^\infty(X,\C^N)$ of finite codimension the equation
\begin{equation}\label{eqintrolocsolv}
Pu=f
\end{equation}
has a local weak solution $u\in\De(X,\C^N)$ in a neighborhood of $M$.
We can also define microlocal solvability at a set in the cosphere bundle,
or equivalently, at a conic set in $T^\ast(X)\smallsetminus 0$, the cotangent bundle of $X$ with the
zero section removed.
By a conic set $K\subset T^\ast(X)\smallsetminus 0$ we mean a set that is
conic in the fiber, that is,
\[
(x,\xi)\in K \quad \Longrightarrow \quad (x,\lambda \xi)\in K \quad \text{for all }\lambda>0.
\]
If, in addition, $\pi_x(K)$ is compact in $X$, where $\pi_x:T^\ast(X)\to X$ is the projection,
then $K$ is said to be compactly based.
Thus, we say that $P$ is solvable at the compactly based cone $K \subset T^{\ast}(X) \smallsetminus 0$
if there is an integer $N_0$ such that for every
$f \in H_{(N_0)}^{\mathrm{loc}}(X,\C^N)$
there exists a $u \in \De(X,\C^N)$ with $K \cap W\! F(Pu-f)=\emptyset$ (see Definition \ref{defrangesystems}).

The famous example due to Hans Lewy~\cite{le}
showed that not all smooth linear differential operators are solvable.
This example led to an extension due to H{\"o}rmander~\cite{ho10,ho11}
in the sense of a necessary condition for a differential equation $P(x,D)u=f$ to have a
solution locally for every $f\in C^\infty$.
In fact (see~\cite[Theorem~6.1.1]{ho0}),
if $\varOmega$ is an open set in $\mathbb{R}^n$, and $P$ is a differential
operator of order $m$ with coefficients in $C^\infty(\varOmega)$ such that the differential equation
$P(x,D)u=f$
has a solution $u\in \De(\varOmega)$ for every $f\in C_0^\infty(\varOmega)$, then
$\{p, \overline{p} \}$ must vanish at every point $(x,\xi)\in \varOmega \times \mathbb{R}^n$ for which
$p(x,\xi)=0$,
where $p$ is the principal symbol of $P$ and
\[
\{ a , b \}=\sum_{j=1}^n \partial_{\xi_j}a \, \partial_{x_j}b- \partial_{x_j} a
\, \partial_{\xi_j} b
\]
denotes the Poisson bracket.

Recall that a scalar
pseudodifferential operator $P$ is of principal type
if the Hamilton vector field $H_p$ of the principal symbol $p$ is not proportional
to the radial vector field $\rho$ when $p=0$, where $H_p :f\mapsto\{p,f\}$ for $f\in C^\infty$
and $\rho$ is given in terms of local coordinates on $T^{\ast}(X) \smallsetminus 0$
by $\xi\partial_\xi$.
For such operators
it was conjectured by Nirenberg and Treves~\cite{nitr} that
local solvability at a compact set $M\subset X$
in the sense of \eqref{eqintrolocsolv}
is equivalent to
condition $(\varPsi)$ on the principal symbol, which means that
there is a neighborhood $Y$ of $M$ such that
\begin{multline}\label{intropsi}
\im ap \text{ does not change sign from $-$ to $+$} \\
\text{along the oriented bicharacteristics of } \re ap
\end{multline}
over $Y$ for any $0\neq a\in C^\infty(T^*(Y)\smallsetminus 0)$.
The oriented bicharacteristics of $\re ap$
are the positive flow-outs of the Hamilton vector field
$H_{\re ap}$ on $\re ap=0$, sometimes referred to as semi-bicharacteristics of $p$.
Note that condition \eqref{intropsi} is invariant
under multiplication of $p$ with non-vanishing factors
and symplectic changes of coordinates. Hence the condition is invariant under
conjugation of $P$ with elliptic
Fourier integral operators.

The necessity of condition $(\varPsi)$ for local solvability of scalar
pseudodifferential operators of principal type was proved by Moyer~\cite{mo} in $1978$
for the two dimensional case and by H{\"o}rmander~\cite{ho40} in $1981$ for the general case.
It was finally shown by
the first author~\cite{de}
in $2006$ that condition $(\varPsi)$ is also sufficient for local
and microlocal solvability for scalar operators of principal type.

For systems, no corresponding conjecture for solvability exists. However, by considering
the case when the principal symbol of a square system $P$ of principal type has
\emph{constant characteristics} (see Definition \ref{def:constantchars}),
the first author~\cite{de2}
showed that local and microlocal solvability is equivalent to condition $(\varPsi)$
on the eigenvalues of the principal symbol. Here we wish to mention that
although not explicitly addressed in~\cite{de2}, one actually finds that
for systems of principal type with constant characteristics,
condition $(\varPsi)$ on the eigenvalues of the principal symbol is necessary
also for semi-global solvability in the sense of~\cite[Theorem~26.4.7]{ho4}.
For easy reference we have included a statement of this result,
see Theorem \ref{thm:solv} below and also the reformulation of the result given in
Corollary \ref{thm:solv1}.

To address a conjecture made by Lewy stipulating that scalar differential operators which fail
to have local solutions are essentially uniquely determined by the range,
H{\"o}rmander~\cite[Chapter~6.2]{ho0} proved that
if $P$ and $Q$ are two first order differential operators with coefficients in $C^\infty(\varOmega)$
and in
$C^1(\varOmega)$, respectively, such that
the equation $P(x,D)u=Q(x,D)f$ has a solution $u\in \De(\varOmega)$
for every $f\in C_0^\infty(\varOmega)$, and $x$ is a point in $\varOmega$ such that
\begin{equation}\label{eq:introeq0.9}
p(x,\xi)=0 , \quad \{ p , \overline{p} \}(x,\xi)\neq 0
\end{equation}
for some $\xi\in \mathbb{R}^n$, then there is a constant $\mu$ such that (at the fixed point $x$)
\[
Q(x,D)=P(x,D) \mu .
\]
This result was generalized to scalar classical pseudodifferential operators
of principal type by the second author, see~\cite[Theorem~2.19]{jw}. It was shown that if
the principal symbol $p$ of $P\in\varPsi_\mathrm{cl}^m(X)$
fails to satisfy condition $(\varPsi)$ along a curve $\gamma$ in
place of the condition given by \eqref{eq:introeq0.9}, and if
the range of $Q\in\varPsi_\mathrm{cl}^k(X)$ is microlocally contained in the range of $P$ at
a cone $K$ containing $\gamma$,
then one can find an operator $E\in\varPsi_\mathrm{cl}^{k-m}(X)$ such that all the terms in the
asymptotic expansion of the symbol of $Q-PE$
vanish of infinite order at every point belonging to
a \emph{minimal bicharacteristic} $\varGamma\subset\gamma$ of $p$.
For the definition of $\varGamma$,
see Section \ref{sec:minimalbicharacteristics} and Definition \ref{dfn:minimal1}
in particular.
It was also shown that one recovers the mentioned result for first order differential
operators, if $Q$ is assumed to have $C^\infty$ coefficients.
The main result of this paper is a generalization of~\cite[Theorem~2.19]{jw}
to systems of principal type with constant
characteristics, see Theorem \ref{thm:bigsystems}.
We shall only consider operators acting on
distributions $\De (X,\C^N)$
with values in $\C^N$ but since the
results are essentially local (see the trivialization given by
Proposition \ref{prop:basis}) and invariant under base changes, they immediately
carry over to operators on sections of vector bundles.

This article was written during a period when the second author
stayed at Kyoto University, supported by the JSPS postdoctoral
fellowship program. The second author thanks
JSPS for its financial support, and wishes to express
his gratitude to Professor Yoshinori Morimoto at
Kyoto University for his hospitality.

\section{Systems of principal type and constant characteristics}
\label{sec:preliminariesforsystems}

\noindent
Let $X$ be a $C^\infty$
manifold of dimension $n$. In what follows, $C$ will be taken to be a new constant
every time unless stated otherwise. We let $\kernel A$ denote the kernel and
$\ran A$ the range of the matrix $A$, and
let $\Li_N=\Li(\C^N,\C^N)$ be the space
of bounded linear maps from $\C^N$ to $\C^N$.

In this section we will introduce the systems that will be the focus of our study.
For a more thorough discussion as well as multiple examples, we refer to~\cite{de2}.
We begin by recalling the definition of a square system of principal type.

\begin{dfn}\label{def:systemprincip}
We say that the $N\times N$ system $w\mapsto P(w)\in C^1(T^\ast(X)\smallsetminus 0)$ is of principal type at $w_0$ if
\begin{equation}\label{invertibletangentvector}
\partial_\nu P(w_0): \kernel P(w_0)\longrightarrow \cokernel P(w_0)=\C^N/\ran P(w_0)
\end{equation}
is bijective for some $\partial_\nu\in T_{w_0}(T^\ast(X)\smallsetminus 0)$,
where $\partial_\nu P(w_0)=\langle \nu,dP(w_0)\rangle$ and the mapping
is given by $u\mapsto \partial_\nu P(w_0)u$ mod $\ran P(w_0)$. We say that
$P\in \varPsi_{\mathrm{cl}}^m(X)$ is of principal type at $w_0$ if the principal symbol
$P_m(w)$ is of principal type at $w_0$.
\end{dfn}

Due to the relation between the dimensions of the kernel and the cokernel
only square systems can be of principal type.
Moreover, $P(w)\in C^1$ is of principal type if and only if the adjoint $P^\ast$ is
of principal type, and if $A(w),B(w)\in C^1$ are invertible and
$P(w)\in C^1$ is of principal type then $APB$ is of principal type, see~\cite[Remark~2.2]{de2}.

Recall that if
\begin{equation}\label{def:radialmultiplier}
M:\R^+\times T^\ast(X)\smallsetminus 0\to T^\ast(X)\smallsetminus 0
\end{equation}
is the $C^\infty$ map acting through multiplication by $t$ in the fiber,
then the radial vector
field $\rho\in T(T^\ast(X)\smallsetminus 0)$ is invariantly described by
\[
\rho f = \frac{d}{dt} M_t^\ast f|_{t=1}, \quad f\in C^1(T^\ast(X)\smallsetminus 0).
\]
Here $M_t(w)=M(t,w)$ and in terms of local coordinates
we have $M_t(w)=(x,t\xi)$ and $\rho(w)=\xi\partial_\xi$ at $w$ if $w=(x,\xi)$, see the discussion
following~\cite[Definition~21.1.8]{ho3}.
Suppose now that $P$ is an $N\times N$ system of principal type at $w_0$ such that
Definition \ref{def:systemprincip} is satisfied for some
$\partial_\nu\in T_{w_0}(T^\ast(X)\smallsetminus 0)$.
If $P$ is homogeneous of
degree $m$, that is, $M_t^\ast P(w)=t^m P(w)$, then $\partial_\nu$ cannot be
proportional to $\rho(w_0)$. Indeed, differentiation gives $\rho P=mP$ in view
of Euler's homogeneity relation, so $u\mapsto\rho P(w_0)u=0$ for all $u\in\kernel P(w_0)$.
Hence $\rho P(w_0):\kernel P(w_0)\to\cokernel P(w_0)$ cannot be invertible
unless $\kernel P(w_0)$ is trivial.

\begin{rmk}\label{specialPT}
For a scalar operator $P$, Definition \ref{def:systemprincip} coincides with the
notion that the principal symbol $p$ of $P$ vanishes of first order on the kernel, that is,
the differential $dp$ of the principal symbol is non-vanishing at the points where $p=0$.
However, in the homogeneous case
one often defines principal type operators so that the Hamilton
vector field $H_p$ of $p$ is not
proportional to
the radial vector field $\rho$.
This is also the definition we shall use for
scalar operators of principal type. Although not apparent from
Definition \ref{def:systemprincip}, this would not be
an inconvenience due to the properties of minimal
bicharacteristics, near which we will do our analysis.
However, we would like to point out that
if $\omega$ is the canonical one form
then we recover the scalar definition of principal type from
Definition \ref{def:systemprincip} applied to scalar symbols
under the additional
condition that the tangent vector $\partial_\nu$ for which the map \eqref{invertibletangentvector}
is invertible also satisfies $\langle \partial_\nu, \omega(w_0)\rangle=0$.
In fact, since $H_p$ is proportional to $\rho$
if and only if $dp$ is proportional to $\omega$, the claim follows.
Note that if $\sigma$ is the symplectic form then
$\langle \rho, \omega\rangle=\sigma(\rho,\rho)=0$, so
this does not exclude multiples of $\rho$
for which we know that Definition \ref{def:systemprincip} does not hold
in the homogeneous case in view of the discussion preceding the remark.
\end{rmk}

The eigenvalues of the principal symbol $P_m$ of an
$N\times N$ system of
classical pseudodifferential operators $P\in \varPsi_{\mathrm{cl}}^m(X)$
are the solutions to the characteristic equation
\begin{equation}\label{eq:syscharacteristicequation}
|P_m(w)-\lambda \mathrm{Id}_N|=0,
\end{equation}
where $|A|$ denotes the determinant of the matrix $A$.
Recall that the algebraic multiplicity of the eigenvalue
$\lambda$ of $P_m(w)$ is the multiplicity of $\lambda$
as a root to equation \eqref{eq:syscharacteristicequation},
while the geometric multiplicity is the dimension of $\kernel (P_m(w)-\lambda\mathrm{Id}_N)$.
If the matrix $P_m(w)$ depends continuously on a parameter $w$,
then the eigenvalues $\lambda(w)$ also depend continuously on $w$.
Following the terminology in~\cite{de2}, such a continuous function
$w\mapsto\lambda(w)$ of eigenvalues will be
referred to as a section of eigenvalues of $P_m(w)$.
We shall usually only write $\lambda(w)$ to signify this property.

One problem with studying systems $P(w)$ is that the eigenvalues are
not very regular in the parameter $w$, generally they depend only
continuously (and eigenvectors measurably) on~$w$, see for
example~\cite[Example~2.16]{de2}.
We will avoid this problem by studying systems with constant characteristics.
Before defining this property we need to introduce some notation.

For an $N \times N$ system $P \in C^\infty ( T^\ast(X)\smallsetminus 0)$
and all integers $k\ge 1$
we define 
\begin{align*}
{\omega}_k(P) & = \{(w,{\lambda}) \in T^\ast(X)\times \C:
\dim\kernel (P(w)-{\lambda}\I_N) \ge k\}, \\ 
{\varOmega}_k(P) & = \{(w,{\lambda}) \in T^\ast(X)\times \C:
\partial_{\lambda}^j |P(w) - {\lambda}\I_N| = 0 \text{ for all } j < k
\}.
\end{align*}
Note that $\omega_k(P)=\varOmega_k(P)=\emptyset$ for all $k>N$ when
$P$ is an $N\times N$ system.
We have $\omega_1(P)=\varOmega_1(P)$ but $\omega_k(P)$ and $\varOmega_k(P)$
could be different when $k>1$ if $P$ is not symmetric.
Clearly, ${\omega}_k(P)$ and ${\varOmega}_k(P)$ are closed sets for any $k \ge 1$,
and
\[
\omega_{k+1}(P)\subset\omega_k(P)\subset\varOmega_k(P)\subset\varOmega_{k-1}(P)
\subset\varOmega_1(P),\quad k>1.
\]
Therefore, we can define
\begin{equation}\label{varXi}
\varUpsilon(P) = \bigcup_{k > 1} \partial {\varOmega}_k(P),
\qquad
\varXi(P) = \bigcup_{k > 1} \partial {\omega}_k(P)\cup\partial {\varOmega}_k(P),
\end{equation}
where $\partial {\omega}_k(P)$ and $\partial {\varOmega}_k(P)$ are the
boundaries in the relative topology of ${\varOmega}_1(P)$.
By the definition we find that the multiplicity of the 
zeros of $|P(w)-{\lambda}\I_N|$ is locally constant on ${\varOmega}_1(P)
\smallsetminus \varUpsilon(P)$ and the dimension $\dim \kernel (P(w) - {\lambda}\I_N)$ is constant on
${\varOmega}_1(P) \smallsetminus (\varXi(P)\smallsetminus\varUpsilon(P))$.
Thus, we find that both the algebraic and the geometric multiplicities
of the eigenvalues of the system $P(w)$ are locally constant on ${\varOmega}_1(P)\smallsetminus
{\varXi}(P)$. Note also
that $\varXi(P)  $  and $ \varUpsilon(P)$ are closed and
nowhere dense in ${\varOmega}_1(P)$ since they are unions of 
boundaries of closed sets. 
Moreover,
\[
(w,{\lambda}) \in {\varXi}(P) \iff (w,\ol {\lambda}) \in {\varXi}(P^\ast)
\]
since $P^\ast - \ol {\lambda}\I_N = (P - {\lambda}\I_N)^\ast$.

\begin{dfn}\label{def:constantchars}
We say that the $N \times N$ system $P(w)$ has constant
characteristics near the set $K$ if 
\[
K \times \{0\} \cap {\varXi}(P) = \emptyset.
\]
If $K$ is a compact set, this means that one can find a neighborhood
$U$ of $K$ and an ${\varepsilon} > 0$ so that $U\times
D_{\varepsilon}(0) \cap {\varXi}(P) = \emptyset$, where
$D_{\varepsilon}(0)$ is the disc at $0$ with radius ${\varepsilon}$.
\end{dfn}

This is a local definition: if the system has constant
characteristics near all points in $K$, then it has constant
characteristics near $K$. 
Note also that if $K$ is compact and $K \times \{0\} \cap \varUpsilon(P) = \emptyset$,
then one can find $U$ and $\varepsilon$ as in Definition \ref{def:constantchars}
such that $U\times D_{\varepsilon}(0) \cap \varUpsilon(P) = \emptyset$.
If $\lambda(w)$ is a section of eigenvalues
of $P(w)$ such that $|\lambda(w)|<\varepsilon$ in $U$
then it is a uniquely defined
$C^\infty$ function there in view of~\cite[Remark~2.4]{de2}.
In particular, if $K$ belongs to the characteristic set
$\varSigma(P)=\{w:|P(w)|=0\}$ of $P(w)$, and $\lambda(w)$
is the section of eigenvalues of $P(w)$ vanishing on $K$,
then after possibly shrinking $U$ we find that $\lambda(w)$
has constant algebraic multiplicity in $U$, so $\lambda(w)\in C^\infty(U)$
is uniquely defined.

When the principal symbol $P_m(w)$ of an
$N\times N$ system of
classical pseudodifferential operators $P\in \varPsi_{\mathrm{cl}}^m(X)$
is homogeneous of degree $m$, then the sections of eigenvalues of $P_m(w)$ are
also homogeneous of degree $m$.

\begin{prop}\label{prop:homogeneouseigenvalues}
Let $X$ be a $C^\infty$ manifold and let $P\in C^\infty(T^\ast(X)\smallsetminus 0)$
be an $N\times N$ system, homogeneous of degree $m$, that is, $M_t^\ast P=t^mP$
where $M$ is the $C^\infty$ map given by \eqref{def:radialmultiplier}
acting through multiplication by $t$ in the fiber.
Then the solutions to the characteristic equation
$|P(w)-\lambda\I_N|=0$
are continuous and homogeneous of degree $m$. Furthermore, the number of
distinct solutions to
\begin{equation}\label{app:chareqt}
|P(M_t(w))-\lambda\I_N|=0
\end{equation}
is a constant function of $t$.
\end{prop}

\begin{proof}
Let $w\mapsto\lambda(w)$ be a solution to
$|P(w)-\lambda\I_N|=0$. Since
$P\in C^\infty$ it follows that $\lambda(w)$ is continuous so we
only have to prove homogeneity. To this end, introduce a Riemannian
metric on $X$ (which by duality allows us to define the unit cotangent bundle), and write
$P(x,\xi)=|\xi|^mp(x,\xi)$ where $p(x,\xi)=P(x,\xi/|\xi|)$
is smooth and homogeneous of degree $0$.
Such functions can be identified with smooth functions on
$S^\ast(X)$, so if $\pi:T^\ast(X)\smallsetminus 0\to S^\ast(X)$
is the projection then we have
$p=\pi^\ast p_s$
for some matrix valued function $p_s\in C^\infty(S^\ast(X),\mathcal L_N)$.
Here $p_s$ depends implicitly on the choice of metric,
but this is of no importance.
For a fixed point
$w$, suppose that $\varrho_1,\ldots,\varrho_\ell$ are the distinct
solutions to $|p_s(\pi(w))-\varrho\I_N|=0$. By the homogeneity of $P$ it follows that
if $w=(x,\xi)$ in local coordinates, then for any $t>0$ we have
\[
0=|M_t^\ast P(w)-M_t^\ast\lambda(w)\I_N|=(t|\xi|)^{mN}|p_s(\pi(w))-(t|\xi|)^{-m}M_t^\ast\lambda(w)\I_N|,
\]
so there exists an integer $k(t)\in\{1,\ldots,\ell\}$ such that
\begin{equation}\label{eq:repformulaforeigenvalues}
M_t^\ast\lambda(w)=(t|\xi|)^{m}\varrho_{k(t)}.
\end{equation}
Since $\lambda(w)$ is continuous
and the eigenvalues $\varrho_k$ are distinct, equation \eqref{eq:repformulaforeigenvalues}
implies that the integer valued map $t\mapsto k(t)$ is locally constant.
Since $\R^+$ is connected, it follows that $k(t)\equiv k$ for some $1\le k\le\ell$.
In particular, $k(t)=k(1)$ for all $t>0$, which yields
\begin{equation*}
M_t^\ast\lambda(w)=(t|\xi|)^{m}\varrho_{k(t)}=t^m(|\xi|^{m}\varrho_{k(1)})=t^mM_1^\ast\lambda(w)
=t^m\lambda(w),
\end{equation*}
so $\lambda(w)$ is homogeneous of degree $m$.

To prove the last statement of the proposition, let $\ell(t)$ be the
number of distinct solutions to \eqref{app:chareqt}. By
the first part of the proof these solutions are homogeneous,
which implies that there are at least $\ell(t)$ distinct solutions
at the point $M_{t'}^\ast(w)$. Thus $\ell(t)\le\ell(t')$.
By symmetry we also have $\ell(t')\le\ell(t)$, which completes the proof.
\end{proof}

\begin{cor}\label{cor:algmult}
Let $P\in C^\infty(T^\ast(X)\smallsetminus 0)$
be an $N\times N$ system, homogeneous of degree $m$, and
let $K\subset T^\ast(X)\smallsetminus 0$ be a compact set. Suppose
that $w\mapsto\lambda(w)$ is a section of eigenvalues of $P$ with
constant algebraic multiplicity for $w$ in
\[
K_\varepsilon=\{w\in T^\ast(X)\smallsetminus 0:\inf_{w_0\in K}|w-w_0|<\varepsilon\},
\]
with distance given in terms of some fixed Riemannian metric.
Then $\lambda(w)$ has constant algebraic multiplicity in the cone
\[
\varGamma=\{M_t(w):t>0, w\in K_\varepsilon\}.
\]
\end{cor}

\begin{proof}
Let $M_t(w)\in\varGamma$
and suppose that the algebraic multiplicity of $\lambda(w)$ equals $k$
for $w\in K_\varepsilon$.
By assumption we then have $|P(w)-\lambda\I_N|=(\lambda-\lambda(w))^k e(w,\lambda)$,
where $e(w,\lambda(w))\ne 0$. Consider now equation
\eqref{app:chareqt}. By the homogeneity of $P$ this equation is equivalent to
$|P(w)-t^{-m}\lambda\I_N|=0$ for any $t>0$. Since the left-hand side equals
$(t^{-m}\lambda-\lambda(w))^k e(w,t^{-m}\lambda)$ and
$\lambda(w)=t^{-m}M_t^\ast \lambda(w)$ by Proposition \ref{prop:homogeneouseigenvalues},
this shows that $\lambda=M_t^\ast\lambda(w)$ is a solution to \eqref{app:chareqt}
of at least multiplicity $k$.
Using homogeneity again we find that for $\lambda=M_t^\ast\lambda(w)$
we have $e(w,t^{-m}\lambda)=e(w,\lambda(w))\ne 0$, which shows that the multiplicity
is precisely $k$. Since $M_t(w)\in\varGamma$ was arbitrary, the proof is complete.
\end{proof}

In view of Corollary \ref{cor:algmult} we shall sometimes permit
us to say that a system $P$ has constant characteristics near a conic
set $K\subset T^\ast(X)\smallsetminus 0$ if it is clear from the context
what we mean.
Suppose now that $P(w)$ is homogeneous of degree $m$ and of principal type
with constant characteristics near a compact set $K\subset T^\ast(X)\smallsetminus 0$
contained in the characteristic set $\varSigma(P)=\{w: |P(w)|=0\}$ of $P(w)$.
Let $\lambda(w)$ be the unique section of eigenvalues of $P(w)$ near $K$
satisfying $\lambda(w)=0$ for $w\in K$.
By Definition \ref{def:constantchars} together with~\cite[Proposition~2.10]{de2}
we then have $d\lambda(w)\ne 0$ in $K$,
which in view of Proposition \ref{prop:homogeneouseigenvalues}
implies that $d\lambda(w)\ne 0$ in a conic neighborhood of $K$.
In particular, this means that for systems of principal type with constant characteristics,
the section of eigenvalues close to the origin is a uniquely defined $C^\infty$
function with non-vanishing differential,
so the semi-bicharacteristics of the eigenvalues are
well defined near the characteristic set $\varSigma(P)$. This makes
the following definition possible.

\begin{dfn}\label{def:conditionpsiforsystems}
We say that the $N\times N$ system $P\in\varPsi_{\mathrm{cl}}^m(X)$ of principal type
and constant characteristics satisfies condition $(\varPsi)$ if the eigenvalues
of the principal symbol satisfies condition $(\varPsi)$.
\end{dfn}

Similarly, by the previous discussion it follows that
the condition that the Hamilton vector field of an eigenvalue $\lambda$
does not have the radial direction when $\lambda=0$ is also well defined.
Under this additional assumption, the section of eigenvalues close to the origin is then a uniquely defined
homogeneous $C^\infty$
function of principal type.
In fact, if Definition \ref{def:systemprincip} is changed to include
the additional condition discussed in the remark following the definition,
then the characterization of systems of principal type given by~\cite[Proposition~2.10]{de2}
takes the following form. This is included only for the sake of completeness and will not
be used here.

\begin{prop}\label{characterization}
Let $P(w)\in C^\infty(T^\ast(X)\smallsetminus 0)$ be an $N\times N$ system such that $|P(w_0)|=0$,
and let $\varUpsilon(P)$ be given by \eqref{varXi}. Assume that
\[
\{w_0\}\times\{0\}\cap\varUpsilon(P)=\emptyset.
\]
Let $\lambda(w)\in C^\infty$ be the unique section
of eigenvalues of $P(w)$ satisfying $\lambda(w_0)=0$. If $\omega$ is the canonical
one form then $P(w)$ satisfies
Definition \ref{def:systemprincip} for some tangent vector
$\partial_\nu\in T_{w_0}(T^\ast(X)\smallsetminus 0)$
such that $\langle\partial_\nu,\omega(w_0)\rangle=0$
if and only if
the Hamilton vector field
$H_\lambda(w_0)$ is not proportional to the radial vector field at $w_0$
and the geometric multiplicity of the
eigenvalue $\lambda$ is equal to the algebraic multiplicity at $w_0$.
\end{prop}

Note that as suggested in the statement of the proposition, the hypotheses
$|P(w_0)|=0$ and $\{w_0\}\times\{0\}\cap\varUpsilon(P)=\emptyset$ imply that
the section of eigenvalues $\lambda(w)$ of $P(w)$ satisfying $\lambda(w_0)=0$
is a uniquely defined $C^\infty$ function in a neighborhood of $w_0$
according to the discussion following Definition \ref{def:constantchars}.

\begin{proof}
Inspecting the beginning of the proof of~\cite[Proposition~2.10]{de2} we conclude that
the same arguments show that $P(w)$ satisfies Definition \ref{def:systemprincip} for some tangent vector
$\partial_\nu\in T_{w_0}(T^\ast(X)\smallsetminus 0)$
such that $\langle\partial_\nu,\omega(w_0)\rangle=0$
if and only if
\[
\partial_\nu^k|P(w_0)|\ne 0,\quad
\langle\partial_\nu,\omega(w_0)\rangle=0,
\quad k=\dim\kernel P(w_0).
\]
Now, if $P(w)$ is of principal type at $w_0$
then the geometric multiplicity $k$ of $\lambda$ is equal to the algebraic multiplicity $m$
at $w_0$ by~\cite[Proposition~2.10]{de2}. Thus
\[
\partial_\nu^m|P(w_0)|\ne 0,\quad |P(w)-\lambda\I_N|=(\lambda(w)-\lambda)^m e(w,\lambda)
\]
for $w$ in a neighborhood of $w_0$ where $e(w,\lambda)\ne 0$.
Setting $\lambda=0$ we obtain $0\ne \partial_\nu^m|P(w_0)|=(\partial_\nu\lambda(w_0))^m e(w_0,0)$.
If $\langle\partial_\nu,\omega(w_0)\rangle=0$
and $d\lambda(w_0)=\mu\omega(w_0)$ at $w_0$
for some $\mu\in\C$, then
$0\ne\partial_\nu\lambda(w_0)=\mu\langle\partial_\nu,\omega(w_0)\rangle=0$,
a contradiction.

To prove sufficiency, we note that if $H_\lambda(w_0)$
is not proportional to the radial vector field at $w_0$ then
we can find a tangent vector $\partial_\nu\in T_{w_0}(T^\ast(X)\smallsetminus 0)$
such that $\langle \partial_\nu,d\lambda(w_0)\rangle\ne 0$
and $\langle \partial_\nu, \omega (w_0)\rangle=0$. But this gives
$\partial_\nu^m|P(w_0)|\ne 0$ where $m$ equals the algebraic and geometric multiplicity
at $w_0$, so by the first paragraph we conclude that
$P(w)$ satisfies Definition \ref{def:systemprincip} for a tangent vector
$\partial_\nu$ such that
$\langle\partial_\nu,\omega(w_0)\rangle=0$.
This completes the proof.
\end{proof}

\section{Minimal bicharacteristics}
\label{sec:minimalbicharacteristics}

The purpose of this section is to recall the geometry
that occurs when condition $(\varPsi)$ is violated.
For a more thorough discussion as well as proofs for
the results below we refer the reader to~\cite[Section~2]{jw},
on which the following review is based.

Let us first fix some terminology. If $\gamma\subset T^\ast(X)$
is a curve with a parametrization $t\mapsto\gamma(t)$ defined (at least) for $a\le t\le b$,
we shall say that $\im qp$
changes sign from $-$ to $+$ on $\gamma$ if
\begin{equation}\label{notintropsi}
\im qp(\gamma(a))<0<\im qp(\gamma(b)).
\end{equation}
If $\gamma |_{[a',b']}$ is the restriction of $\gamma$ to $[a',b']$
and we have
\begin{itemize}
\item[i)] $\im qp (\gamma(t))= 0$ for $a'\le t\le b'$,
\item[ii)] for every $\varepsilon > 0$
one can find $a' - {\varepsilon} < s_- < a'$ and $b' < s_+ < b' + {\varepsilon}$ such that
$\im qp(\gamma(s_-)) < 0 < \im qp(\gamma(s_+))$,
\end{itemize}
then we shall say that $\im qp$ \emph{strongly}
changes sign from $-$ to $+$ on $\gamma |_{[a',b']}$.
If $p$ and $q$ are smooth homogeneous functions and
$\gamma$ is a bicharacteristic of $\re qp$ where $q\ne 0$ and
\eqref{notintropsi} holds, then we can always find a
subinterval of $\gamma$ where $\im qp$ strongly changes sign
from $-$ to $+$ by~\cite[Lemma~2.5]{jw}.

Consider now the case where $p\in C^\infty(T^\ast(\R^n)\smallsetminus 0)$
satisfies $\re p=\xi_1$.
If $\gamma = I \times \{w_0\}$, $I= [a,b]$, we shall by $|\gamma|$
denote the usual arc length in $\mathbb{R}^{2n}$, so that $|\gamma|=b-a$.
Furthermore, we will assume that all curves are bicharacteristics of $\re p=\xi_1$, that
is, $w_0=(x',0,\xi')\in\mathbb{R}^{2n-1}$.
We shall then employ the following notation.

\begin{dfn}
Let $\gamma = [a,b] \times \{w_0\}$, and let ${\gamma}_j = [a_j, b_j] \times \{w_j\}$. If
$\lim_{j\to\infty}w_j = w_0$,
$\liminf_{j \to \infty} a_j \ge a$ and $\limsup_{j \to \infty} b_j \le b$,
then we shall write ${\gamma}_j \dashrightarrow {\gamma}$ as $j\to\infty$.
If in addition $\lim_{j \to \infty} a_j=a$ and $\lim_{j \to \infty} b_j=b$ then we
shall write $\gamma_j\to\gamma$ as $j\to\infty$.
\end{dfn}

\begin{dfn}\label{def:minimal02}
If $\gamma$ is a bicharacteristic of $\re p=\xi_1$ and there exists a sequence
$\{\gamma_j\}_{j=1}^\infty$ of bicharacteristics of $\re p$ such that $\im p$
strongly changes sign from $-$ to $+$ on $\gamma_j$
for all $j$
and $\gamma_j\dashrightarrow\gamma$ as $j\to\infty$, we set
\begin{equation*}
L_p({\gamma}) = 
\inf\{\liminf_{j\to\infty} |\gamma_j| : \gamma_j \dashrightarrow \gamma \,
\text{ as }j\to\infty\},
\end{equation*}
where the infimum is taken over all such sequences.
We shall write $L_p({\gamma})\ge 0$ to signify the existence of such a sequence $\{\gamma_j\}_{j=1}^\infty$.
\end{dfn}

Note that the definition of $L_p(\gamma)$ corresponds to what is denoted by $L_0$ in~\cite[p.~97]{ho4},
when $\gamma = [a,b] \times \{w_0\}$ is given by
\begin{equation*}
a\leq x_1 \leq b, \quad x'=(x_2, \ldots, x_n)=0, \quad \xi=\varepsilon_n,
\end{equation*}
and $\im p(a,w_0)<0<\im p(b,w_0)$.
For a proof of this claim, see the remark following~\cite[Definition~2.9]{jw}.
Here $\varepsilon_n=(0,\ldots,0,1)\in \mathbb{R}^n$, and we shall
in what follows write $\xi^0$ in place of $\varepsilon_n'$.
Note also that
if $L_p(\gamma)$ exists, then $L_p(\gamma)\le |\gamma|$ by definition. Moreover,
if $\im p$ strongly changes sign from $-$ to $+$ 
on ${\gamma}$ then it is easy to see that the conditions of
Definition \ref{def:minimal02} are satisfied.

We now recall the definition of a minimal bicharacteristic.

\begin{dfn}\label{dfn:minimal1}
Let $I\subset\mathbb{R}$ be a compact interval possibly reduced to a point and
let $\tilde{\gamma}:I\to T^\ast(X)\smallsetminus 0$
be a characteristic point or a compact one dimensional bicharacteristic interval
of the homogeneous function $p\in C^\infty(T^\ast(X)\smallsetminus 0)$.
Suppose that there exists a function $q\in C^\infty(T^\ast(X)\smallsetminus 0)$
and a $C^\infty$
homogeneous canonical transformation $\chi$ from an open conic neighborhood $V$ of
\[
\varGamma=\{(x_1,0,\varepsilon_n): x_1\in I\}
\subset T^\ast (\mathbb{R}^n)
\]
to an open conic neighborhood
$\chi(V)\subset T^\ast(X)\smallsetminus 0$
of $\tilde{\gamma}(I)$ such that
\begin{itemize}
\item[(i)] $\chi(x_1,0,\varepsilon_n)=\tilde{\gamma}(x_1)$ and $\re \chi^\ast(qp )=\xi_1$
in $V$,
\item[(ii)] $L_{\chi^\ast(qp)}(\varGamma)=|\varGamma|$.
\end{itemize}
Then we say that $\tilde{\gamma}(I)$ is a minimal characteristic point
or a minimal bicharacteristic interval
if $|I|=0$ or $|I|>0$,
respectively.
\end{dfn}

The definition of the arclength is of course dependent of the choice of Riemannian
metric on $T^\ast (\mathbb{R}^n)$. However, since we are only using the arclength
to compare curves where one is contained within the other and both are parametrizable through condition (i),
the results here and Definition
\ref{dfn:minimal1} in particular are independent of the chosen metric.

Some comments on the implications of Definition \ref{dfn:minimal1}
are in order. First, note that condition (i) implies that
$q\ne 0$ and $\re H_{qp}\ne 0$ on $\tilde{\gamma}$,
and that by definition,
a minimal bicharacteristic interval
is a compact one dimensional bicharacteristic interval
(see~\cite[Definition~26.4.9]{ho4}).
If $\im qp$ changes sign from $-$ to $+$ on a bicharacteristic
$\gamma\subset T^\ast(X)\smallsetminus 0$
of $\re qp$ where $q\neq 0$, then we can always find a minimal characteristic point $\tilde{\gamma}
\in\gamma$ or a minimal bicharacteristic interval
$\tilde{\gamma}\subset\gamma$. In the language of~\cite[Section~26.4]{ho4},
$\tilde\gamma$ is the subset of $\gamma$ with the property that 
$\im qp$ changes sign from $-$ to $+$ on bicharacteristics of $\re qp$ 
arbitrarily close to $\tilde\gamma$. For a proof of this fact,
see~\cite[p.~97]{ho4} or the discussion preceding~\cite[Proposition~2.12]{jw}.
In fact, we have the following result.

\begin{prop}\label{prop:minimal03}
Let $\gamma=[a,b]\times \{w_0\}$ be a bicharacteristic of $\re p=\xi_1$,
and assume that $L(\gamma)\ge 0$.
Then there exists a minimal characteristic point $\varGamma\in\gamma$ of $p$ or
a minimal bicharacteristic interval $\varGamma\subset\gamma$ of $p$ of length $L(\gamma)$
if $L(\gamma)=0$ or $L(\gamma)>0$, respectively.
If $\varGamma=[a_0,b_0]\times \{w_0\}$ and $a_0<b_0$, that is, $L(\gamma)>0$, then
\begin{equation*}
\im p_{(\alpha)}^{(\beta)}(t,w_0)=0
\end{equation*}
for all $\alpha, \beta$ with
$\beta_1=0$ if
$a_0\le t\le b_0$.
Conversely, if $\gamma$ is a minimal characteristic point or
a minimal bicharacteristic interval
then $L(\gamma)=|\gamma|$.
\end{prop}

\begin{proof}
See the proof of~\cite[Proposition~2.12]{jw}.
\end{proof}

Keeping the notation from Definition \ref{dfn:minimal1},
we note in view of Proposition \ref{prop:minimal03} that
condition (ii)
implies that
there exists a sequence $\{\varGamma_j\}_{j=1}^\infty$
of bicharacteristics of $\re \chi^\ast (qp)$ on which $\im\chi^\ast (qp)$ strongly changes
sign from $-$ to $+$,
such that $\varGamma_j\to\varGamma$
as $j\to\infty$.
By our choice of terminology, the sequence $\{\varGamma_j\}_{j=1}^\infty$
may simply be a sequence of points when $L(\varGamma)=0$. Conversely, if
$\{\varGamma_j\}_{j=1}^\infty$
is a point sequence then $L(\varGamma)=0$.
Also note that if $\tilde{\gamma}(I)$ is minimal, and
condition (i) in Definition \ref{dfn:minimal1} is satisfied
for some other choice of maps $q',\chi'$, then condition (ii)
also
holds for $q',\chi'$;
in other words,
\[
L_{\chi^\ast(qp)}(\varGamma)=|\varGamma|=L_{(\chi')^\ast(q'p)}(\varGamma).
\]
This follows
by an application of Proposition \ref{prop:minimal03}
together with~\cite[Lemma~26.4.10]{ho4}.
It is then also clear that $\tilde{\gamma}(I)$ is a minimal characteristic point
or a minimal bicharacteristic interval of 
the homogeneous function $p\in C^\infty(T^\ast(X)\smallsetminus 0)$ if and only if
$\varGamma(I)$ is a minimal characteristic point
or a minimal bicharacteristic interval of 
$\chi^\ast(qp)\in C^\infty(T^\ast(\mathbb{R}^n)\smallsetminus 0)$
for any maps $q$ and $\chi$ satisfying condition (i) in Definition \ref{dfn:minimal1}.

\begin{dfn}\label{dfn:minimal2}
A minimal bicharacteristic
interval $\varGamma=[a_0,b_0]\times \{w_0\}\subset T^\ast(\mathbb{R}^n)\smallsetminus 0$
of the homogeneous function $p=\xi_1+i\im p$ of degree $1$
is said to be $\varrho$-minimal if there exists a
$\varrho\ge 0$ such that $\im p$ vanishes in a neighborhood of
$[a_0+\kappa,b_0-\kappa]\times \{w_0\}$
for any $\kappa>\varrho$.
\end{dfn}

By a $0$-minimal bicharacteristic interval $\varGamma$ we thus mean a minimal bicharacteristic interval such
that the imaginary part vanishes in a neighborhood of any proper closed subset of $\varGamma$. Note that
this does not hold for minimal bicharacteristic intervals in general.
However, the following result does hold, which concludes this section.

\begin{thm}\label{thm:minimal07}
If $\varGamma$ is a minimal bicharacteristic interval
in $T^\ast(\mathbb{R}^n)\smallsetminus 0$ of the homogeneous function $p=\xi_1+i\im p$
of degree $1$, where the imaginary part is independent of $\xi_1$, then
there exists a sequence $\{\varGamma_j\}_{j=1}^\infty$ of $\varrho_j$-minimal
bicharacteristic intervals of $p$ such that $\varGamma_j\to\varGamma$
and $\varrho_j\to 0$ as $j\to\infty$.
\end{thm}

\begin{proof}
See the proof of~\cite[Theorem~2.18]{jw}.
\end{proof}

\section{Solvability and microlocal inclusion relations}
\label{sec:mainresultsforsystems}

If $u=(u_j)$ and $v=(v_j)$ are vectors in $\C^N$ with
$u_j$ and $v_j$ in $L^2(X,\C)$ for $1\le j\le N$, let
\begin{equation}\label{L2pairing}
(u,v)_{L^2(X,\C^N)}=\sum_{j=1}^N (u_j,v_j)
\end{equation}
where $(\phantom{i},\phantom{i})$ denotes the usual scalar product on $L^2(X,\C)$.
Recall that the Sobolev space $H_{(s)}(X,\C)$, $s\in\mathbb{R}$, is a local space,
that is, if $\varphi\in C_0^\infty(X,\C)$ and $\psi\in H_{(s)}(X,\C)$
then $\varphi \psi \in H_{(s)}(X,\C)$, and the corresponding operator
of multiplication is continuous. If $\|\phantom{i}\|_{(s)}$ is the usual
norm on $H_{(s)}(X,\C)$
we shall with abuse of notation let $H_{(s)}(X,\C^N)$ be the space of distributions
$u=(u_j)\in\De(X,\C^N)$ such that $u_j\in H_{(s)}(X,\C)$
for $1\le j\le N$, equipped with the norm
\[
\|u\|_{(s)}=\Big(\sum_{j=1}^N \|u_j\|_{(s)}^2\Big)^{1/2}.
\]
Thus we can define
\[
H_{(s)}^{\mathrm{loc}}(X,\C^N)=\{u\in \De(X,\C^N) : 
\varphi u\in H_{(s)}(X,\C^N), \forall \varphi \in C_0^\infty(X,\C)\}.
\]
This is a Fréchet space, and its dual with respect to the pairing
\eqref{L2pairing} is
\[
H_{(-s)}^{\mathrm{comp}}(X,\C^N)
=H_{(-s)}^{\mathrm{loc}}(X,\C^N)\cap \E'(X,\C^N).
\]
Recall also that the wave front set of $u=(u_j)\in\De(X,\C^N)$ is defined
as the union of $W\! F(u_j)$. For a system $A$ of pseudodifferential operators in $X$
we shall as usual let $W\! F(A)$ be the smallest closed conic set in $T^\ast(X)\smallsetminus 0$
such that $A\in\varPsi^{-\infty}$ in the complement.

\begin{dfn}\label{defrangesystems}
If $K \subset T^{\ast}(X) \smallsetminus 0$ is a compactly based cone we
shall say that the range of the $N\times N$ system $Q\in \varPsi_{\mathrm{cl}}^k(X)$
is microlocally contained in the range of the $N\times N$ system $P\in \varPsi_{\mathrm{cl}}^m(X)$ at $K$ if there
exists an integer $N_0$ such that for every
$f\in H_{(N_0)}^{\mathrm{loc}}(X,\C^N)$ one can find a $u\in \De(X,\C^N)$
with $W\! F(Pu-Qf)\cap K = \emptyset$.
\end{dfn}

\noindent If $\I_N\in\varPsi_{\mathrm{cl}}^0(X)$ is the
identity $\I_N:u\mapsto u\in\De(X,\C^N)$ then
we obtain from Definition \ref{defrangesystems}
the definition of microlocal solvability for a system of
pseudodifferential operators
(see~\cite[Definition~26.4.3]{ho4} and the discussion following equation (1.1) in~\cite{de})
by setting $Q=\I_N$.
Thus, the range of the identity is microlocally contained in the range of $P$ at $K$
if and only if $P$ is microlocally solvable at $K$.
Note also that if $P$ and $Q$ satisfy Definition \ref{defrangesystems} for some integer $N_0$,
then due to the inclusion
\[H_{(t)}^{\mathrm{loc}}(X,\C^N)\subset H_{(s)}^{\mathrm{loc}}(X,\C^N), \quad \textrm{if } s<t,
\]
the statement also holds for any integer $N' \geq N_0$. Hence $N_0$ can always be assumed to be positive.
Furthermore, the property is preserved if $Q$ is composed with a properly supported $N\times N$
system $Q_1\in \varPsi_{\mathrm{cl}}^{k'}(X)$ from the right.
Indeed,
let $g$ be an arbitrary element in
$H_{(N_0+k')}^{\mathrm{loc}}(X,\C^N)$. Then $f=Q_1 g\in H_{(N_0)}^{\mathrm{loc}}(X,\C^N)$
since $Q_1$ is continuous
\[
Q_1:H_{(s)}^{\mathrm{loc}}(X,\C^N)\rightarrow H_{(s-k')}^{\mathrm{loc}}(X,\C^N)
\]
for every $s\in\mathbb{R}$,
so by Definition \ref{defrangesystems}
there exists a $u\in \De(X,\C^N)$ with $W\! F(Pu-Qf)\cap K = \emptyset$.
Hence the range of $QQ_1$
is microlocally contained in the range of $P$ at $K$ with the integer
$N_0$ replaced by $N_0+k'$.

The property given by Definition \ref{defrangesystems} is also
preserved under composition of both $P$ and $Q$ with a properly supported $N\times N$
system from the left. In view of \eqref{defsys} this follows immediately
from the fact that properly supported
scalar pseudodifferential operators are microlocal, that is,
\[
W\! F(Au)\subset W\! F(u)\cap W\! F(A), \quad u\in \De(X).
\]

Just as microlocal solvability of a pseudodifferential operator $P$ implies an a priori estimate for the adjoint
$P^\ast$, we have the following result for systems satisfying Definition \ref{defrangesystems}.

\begin{lem}\label{lemrange1}Let $K \subset T^{\ast}(X) \smallsetminus 0$ be a compactly based cone. Let
$Q\in \varPsi_{\mathrm{cl}}^k(X)$ and $P\in \varPsi_{\mathrm{cl}}^m(X)$
be properly supported $N\times N$ systems such that
the range of $Q$ is microlocally contained in the range of $P$ at $K$. If $Y\Subset X$ satisfies $K\subset
T^*(Y)$ and if $N_0$ is the integer in Definition \ref{defrangesystems}, then for every positive integer $\kappa$
we can find a constant $C$, a positive integer $\nu$ and a properly supported $N\times N$ system
$A$ with $W\! F(A)\cap K= \emptyset$ such that
\begin{equation}\label{rangeeq1}
\|Q^*v\|_{(-N_0)}\leq C(\|P^*v\|_{(\nu)}+\|v\|_{(-N_0-\kappa-n)}+\|Av\|_{(0)})
\end{equation}for all $v\in C_0^\infty(Y)$.
\end{lem}

By replacing the range $\C$ by $\C^N$, the proof of
the corresponding result for the scalar case (see~\cite[Lemma~2.3]{jw})
can be used without additional changes to prove Lemma \ref{lemrange1}.
We omit the details.
Note also that
since \eqref{rangeeq1} holds for any $\kappa$,
it is actually superfluous to include the dimension $n$ in the norm
$\|v\|_{(-N_0-\kappa-n)}$. However, for our purposes, it turns out that this is the most convenient formulation.

We will need the following analogue of~\cite[Proposition~26.4.4]{ho4}.
Since the proof again is the same as for the corresponding result for scalar operators,
we refer to the notation and proof of~\cite[Proposition~2.4]{jw} for details.

\begin{prop}\label{prop.26.4.4} Let $K\subset T^\ast(X)\smallsetminus 0$ and
$K'\subset T^\ast(Y)\smallsetminus 0$ be compactly based cones and let $\chi$ be a
homogeneous symplectomorphism from a conic neighborhood of $K'$ to one of $K$ such that
$\chi(K')=K$. Let $A\in I^{m'}(X\times Y, \varGamma')$ and $B\in I^{m''}(Y\times X,(\varGamma^{-1})')$
where $\varGamma$ is the graph of $\chi$, and assume that the $N\times N$ systems
$A$ and $B$ are properly supported and
non-characteristic at the restriction of the graphs of $\chi$ and $\chi^{-1}$ to $K'$ and to $K$ respectively,
while $W\! F'(A)$ and $W\! F'(B)$ are contained in small conic neighborhoods. Then the range of the $N\times N$
system $Q$ of pseudodifferential operators in $X$ is microlocally contained in the range of the $N\times N$ system
$P$ of pseudodifferential operators in $X$ at $K$ if and only if the range of the system
$BQA$ in $Y$ is microlocally contained in the range of the system
$BPA$ in $Y$ at $K'$.
\end{prop}

It will be convenient to record the following result,
concerning necessary conditions for semi-global solvability
for systems of principal type and constant characteristics,
using the notion of minimal bicharacteristics. Note that
this theorem therefore in a sense corresponds to~\cite[Theorem~26.4.7$'$]{ho4}
in the scalar case.

\begin{thm}\label{thm:solv}
Let $P\in \varPsi_{\mathrm{cl}}^m(X)$
be a properly supported $N\times N$ system of pseudodifferential operators
of principal type
in the open conic set $\varOmega\subset T^{\ast}(X) \smallsetminus 0$.
Let $P_m$ be the homogeneous principal symbol
of $P$, and let $I=[a_0,b_0]\subset \mathbb{R}$ be a compact interval possibly reduced
to a point. Let
$\gamma:I\to \varOmega$ be a curve
belonging to the characteristic set $\varSigma(P_m)$ of $P_m$, and suppose
that $P$ has constant characteristics near $\gamma(I)$.
If $\lambda(w)$ is the unique section of eigenvalues of $P_m(w)$ satisfying $\lambda\circ\gamma = 0$,
and $\gamma$ is either
\begin{enumerate}
\item[(a)] a minimal characteristic point of $\lambda(w)$,
or
\item[(b)] a minimal bicharacteristic interval of $\lambda(w)$
with injective regular projection in $S^\ast (X)$,
\end{enumerate}
then $P$ is not solvable at the cone generated by $\gamma(I)$.
\end{thm}

We wish to point out that although case (b) is not explicitly treated in~\cite{de2},
this result is essentially contained
in~\cite[Theorem~2.7]{de2}.
In fact, for systems of principal type and constant
characteristics,~\cite[Theorem~2.7]{de2} says that
condition \eqref{intropsi} is equivalent to
microlocal solvability near a point $w_0 \in T^{\ast}(X) \smallsetminus 0$
under the additional assumption that the Hamilton vector field $H_\lambda$
of the section of eigenvalues of $P_m(w)$ near $w_0$
satisfying $\lambda(w_0)=0$ does not have the radial direction at $w_0$. 
If $\gamma(I)$ satisfies property (a),
then Definition \ref{dfn:minimal1} implies that $H_\lambda$
is not proportional to the radial vector field at $\gamma(I)$ and that \eqref{intropsi}
cannot hold in any neighborhood of $\gamma(I)$, and since the wave front set
is conic by definition it follows by~\cite[Theorem~2.7]{de2}
that $P$ is not solvable at the cone generated by $\gamma(I)$.
Hence, it only remains to verify Theorem \ref{thm:solv} in the case when
$\gamma$ satisfies property (b). However, note that after locally preparing the system
$P$ to a suitable normal form by means of~\cite[Lemma~4.1]{de2}, the necessity part of~\cite[Theorem~2.7]{de2}
is proved by repetition of the H{\"o}rmander--Moyer proof of the necessity of
condition $(\varPsi)$ for semi-global solvability for scalar operators (see
the proof of~\cite[Theorem~26.4.7]{ho4}). By for example extending the preparation
result~\cite[Lemma~4.1]{de2} to a neighborhood of a one dimensional bicharacteristic interval as
discussed in Section \ref{preparation} below,
the same arguments therefore show that condition $(\varPsi)$ is necessary
also for semi-global solvability for systems of principal type and
constant characteristics. For completeness, we have
included a short proof of Theorem \ref{thm:solv}, which can be found in Section \ref{proofrange}.

We also mention that if $P\in\varPsi_{\mathrm{cl}}^m(X)$
is an $N\times N$ system of principal type
and constant characteristics that is
not microlocally solvable
in any neighborhood of a point $w_0\in T^\ast(X)\smallsetminus 0$,
and the Hamilton vector field $H_\lambda$ of the
section of eigenvalues of $P_m(w)$ near $w_0$
satisfying $\lambda(w_0)=0$ does not have the radial direction at $w_0$,
then $\lambda(w)$ fails to satisfy condition $(\varPsi)$ in every neighborhood of
$w_0$ by~\cite[Theorem~2.7]{de2}. In view
of the alternative version of condition \eqref{intropsi}
given by~\cite[Theorem~26.4.12]{ho4}, it is then easy to see
using~\cite[Theorem~21.3.6]{ho3} and~\cite[Lemma~26.4.10]{ho4}
that $w_0$ is a minimal characteristic point of $\lambda(w)$.

If $\gamma$ is a minimal
bicharacteristic interval of a function $\lambda(w)$ of principal type
such that $\gamma$ is contained in a curve
along which $\lambda(w)$ fails to satisfy condition \eqref{intropsi},
then $\gamma$ has injective regular projection in $S^\ast (X)$ by
the proof of~\cite[Theorem~26.4.12]{ho4}. 
Since solvability at a conic set $K\subset T^\ast(X)\smallsetminus 0$ implies solvability
at any smaller closed cone, 
the discussion preceding Proposition \ref{prop:minimal03}
therefore yields the following corollary to Theorem \ref{thm:solv},
corresponding to~\cite[Theorem~26.4.7]{ho4}.

\begin{cor}\label{thm:solv1}
Let $P\in \varPsi_{\mathrm{cl}}^m(X)$
be a properly supported $N\times N$ system of pseudodifferential operators
of principal type
in the open conic set $\varOmega\subset T^{\ast}(X) \smallsetminus 0$.
Let $P_m$ be the homogeneous principal symbol
of $P$, and let $I=[a_0,b_0]\subset \mathbb{R}$ be a compact interval not
reduced to a point. Let
$\gamma:I\to \varOmega$ be a curve containing a point $\gamma(t_0)\in\varSigma(P_m)$, and suppose
that $P$ has constant characteristics near $\gamma(I)$. If $\varepsilon>0$
is the number given by Definition \ref{def:constantchars}
and $\lambda(w)$ is the unique section of eigenvalues of $P_m$ satisfying $\lambda(\gamma(t_0))=0$,
assume that $|\lambda\circ\gamma(t)|<\varepsilon$ for $a_0\le t\le b_0$ so that
$\lambda(w)$ is a uniquely defined $C^\infty$ function in a neighborhood of $\gamma(I)$.
If there is a homogeneous $C^\infty$ function $q$ in $T^{\ast}(X) \smallsetminus 0$
such that $\gamma$ is a
bicharacteristic interval of $\re q\lambda$ where $\re H_{q\lambda}\ne 0$ and
\[
\im q\lambda(\gamma(a_0))<0<\im q\lambda(\gamma(b_0)),
\]
then $P$ is not solvable at the cone generated by $\gamma(I)$.
\end{cor}

We now proceed to the main result of the paper,
generalizing~\cite[Theorem~2.19]{jw} to systems of principal type
and constant characteristics for which Theorem \ref{thm:solv} implies non-solvability.

\begin{thm}\label{thm:bigsystems}
Let $K \subset T^{\ast}(X) \smallsetminus 0$ be a compactly based cone. Let
$P\in \varPsi_{\mathrm{cl}}^m(X)$ and $Q\in \varPsi_{\mathrm{cl}}^{k}(X)$
be properly supported $N\times N$ systems of pseudodifferential operators such that
the range of $Q$ is microlocally contained in the range of $P$ at $K$, where $P$ is system of principal type
and constant characteristics near $K$.
Let $P_m$ be the homogeneous principal symbol
of $P$, and let $I=[a_0,b_0]\subset \mathbb{R}$ be a compact interval possibly reduced
to a point.
Suppose that
$\gamma:I\to T^\ast(X)\smallsetminus 0$ belongs to the characteristic set
$\varSigma(P_m)$
of $P_m$ and that $K$ contains a conic neighborhood of $\gamma(I)$. If
$\lambda(w)$ is the unique section of eigenvalues of $P_m(w)$ satisfying $\lambda\circ\gamma = 0$,
and $\gamma$ is either
\begin{enumerate}
\item[(a)] a minimal characteristic point of $\lambda(w)$,
or
\item[(b)] a minimal bicharacteristic interval of $\lambda(w)$
with injective regular projection in $S^\ast (X)$,
\end{enumerate}
then there exists an $N\times N$ system
$E\in \varPsi_{\mathrm{cl}}^{k-m}(X)$ such that the terms in the asymptotic
expansion of the symbol of $Q-PE$  
vanish of infinite order at $\gamma(I)$.
\end{thm}

Note that when proving Theorem \ref{thm:bigsystems} we
may assume that $P$ and $Q$ have the same order.
In fact, let $Q_1\in\varPsi_{\mathrm{cl}}^{m-k}(X)$ be a properly supported,
elliptic $N\times N$ system. By the discussion following Definition \ref{defrangesystems}
we have that the range of $QQ_1$ is microlocally contained in the range of $P$ at $K$.
None of the other assumptions in Theorem \ref{thm:bigsystems} are affected by this composition,
so suppose that the theorem is proved for operators of the same order. Since both
$P$ and $QQ_1$ have order $m$, we can then find a system $E\in\varPsi_{\mathrm{cl}}^{0}(X)$
such that all the terms in the asymptotic expansion of the symbol of $QQ_1-PE$
vanish of infinite order at $\gamma(I)$. If $Q_1^{-1}\in\varPsi_{\mathrm{cl}}^{k-m}(X)$ is a properly supported
parametrix of $Q_1$, the calculus then gives that all the terms in the asymptotic
expansion of the symbol of
\[
(QQ_1-PE)\circ Q_1^{-1}\equiv Q-PEQ_1^{-1}\quad\text{mod }\varPsi^{-\infty}
\]
vanish of infinite order at $\gamma(I)$. Thus Theorem \ref{thm:bigsystems}
holds with $E$ replaced by $EQ_1^{-1}\in\varPsi_{\mathrm{cl}}^{k-m}(X)$.

We postpone the proof of Theorem \ref{thm:bigsystems} and instead show
that Theorem \ref{thm:bigsystems} has applications to scalar non-principal type
pseudodifferential operators. For a similar example related to
solvability, see~\cite[Theorem~2.9]{de2}. If $L$ is a properly supported
scalar operator we shall in this context let $L^m$ denote
the composition $L\circ\ldots\circ L$ with $L$ occurring $m$ times, while
$L^0$ is understood to be the identity $\I\in\varPsi_{\mathrm{cl}}^0(X)$.

\begin{thm}\label{thm:application}
Let $K \subset T^{\ast}(X) \smallsetminus 0$ be a compactly based cone. Let
$L\in\varPsi_{\mathrm{cl}}^1(X)$ be a properly supported scalar operator of principal type
near $K$, and let $A_j\in \varPsi_{\mathrm{cl}}^0(X)$ for $0\le j<N$ be
properly supported scalar operators. If $P\in\varPsi_{\mathrm{cl}}^N(X)$ is the operator
\[
Pu=L^N u+\sum_{j=0}^{N-1} A_jL^j u,
\]
let $Q\in\varPsi_{\mathrm{cl}}^k(X)$ be properly supported and assume that
the range of $Q$ is microlocally contained in the range of $P$ at $K$.
Let $w\mapsto\lambda(w)$ be the homogeneous principal symbol
of $L$, and let $I=[a_0,b_0]\subset \mathbb{R}$ be a compact interval possibly reduced
to a point. Suppose that $K$ contains a conic neighborhood of $\gamma(I)$, where
$\gamma:I\to T^\ast(X)\smallsetminus 0$ is either
\begin{enumerate}
\item[(a)] a minimal characteristic point of $\lambda(w)$,
or
\item[(b)] a minimal bicharacteristic interval of $\lambda(w)$
with injective regular projection in $S^\ast (X)$.
\end{enumerate}
Then there exists a properly supported scalar operator
$E\in \varPsi_{\mathrm{cl}}^{k-1}(X)$ such that the terms in the asymptotic sum of the symbol of $Q-PE$  
vanish of infinite order at $\gamma(I)$.
\end{thm}

\begin{proof}
This is a standard reduction to a first order system. If $\mathscr Q$ is the $N\times N$
system given by the block form
\[
\mathscr Q=\left( \begin{array}{cc} 0 & 0\\
0 & Q \end{array} \right)\in\varPsi_{\mathrm{cl}}^k(X),
\]
then the range of $\mathscr Q$
is microlocally contained in the range of $\mathscr P$ at $K$,
where
\[
\mathscr P=\left( \begin{array}{ccccc} L & -1 & 0 & \ldots & 0 \\
0 & L & -1 & \ldots & 0 \\
\vdots & \vdots & \vdots & \ddots & \vdots \\
A_0 & A_1 & A_2 & \ldots & A_{N-1}+L
\end{array} \right)\in\varPsi_{\mathrm{cl}}^1(X).
\]
Indeed, if $N_0$ is the integer given by Definition \ref{defrangesystems},
let $f\in H_{(N_0)}^\mathrm{loc}(X,\C_N)$ be given by $f={}^t(f_1,\ldots,f_N)$. Then
we can find a scalar distribution $u\in\De (X)$ such that
$W\! F (Pu-Qf_N)\cap K=\emptyset$. Now let $v_{j+1}=L^j u$
for $0\le j < N$ and set $v={}^t(v_1,\ldots,v_N)$. Then
$\mathscr Q f={}^t(0,\ldots,0,Q f_N)$ and
$\mathscr P v={}^t(0,\ldots,0,P u)$, which proves the claim.
Since $\lambda(w)$ is the only section of eigenvalues of the principal symbol
of $\mathscr P$, we find by an application of~\cite[Proposition~2.10]{de2}
that $\mathscr P$ is a system of principal type and constant characteristics
near $K$. By Theorem \ref{thm:bigsystems} there is an $N\times N$ system
$\mathscr B=(B_{jk})\in\varPsi_{\mathrm{cl}}^{k-1}(X)$
such that the terms in the asymptotic expansion of the symbol
of $\mathscr Q-\mathscr P\circ \mathscr B$ vanish of infinite order at $\gamma(I)$.
This means that the terms in the asymptotic expansions of the symbols of
\begin{enumerate}
\item[(i)] $Q-A_0 B_{1N}-\ldots -A_{N-1} B_{NN}-LB_{NN}$,
\item[(ii)] $L B_{jN}-B_{(j+1)N},\quad 1\le j<N$,
\end{enumerate}
vanish of infinite order at $\gamma(I)$, which implies that the same
holds for $Q-P B_{1N}$. Indeed, write
\begin{align*}
P B_{1N}&=L^{N-1}(LB_{1N}-B_{2N})+\ldots+L(LB_{(N-1)N}-B_{NN})+LB_{NN}\\
&\phantom{=}+A_0B_{1N}+\sum_{j=1}^{N-1} A_j
\Big(B_{(j+1)N}+ \sum_{\ell=1}^j L^{j-\ell}(LB_{\ell N}-B_{(\ell+1)N})\Big)\\
&=LB_{NN}+\sum_{j=0}^{N-1} A_{j}B_{(j+1)N} + R,
\end{align*}
where $R$ in view of (ii) and the calculus has a symbol with an
asymptotic expansion whose terms vanish of infinite order at $\gamma(I)$.
Since
\[
Q-PB_{1N}=Q-LB_{NN}-\sum_{j=1}^N A_{j-1}B_{jN}- R,
\]
the result now follows by (i) by setting $E=B_{1N}$. This completes the proof.
\end{proof}

Keeping the notation from Theorem \ref{thm:application} and its proof,
we remark that by comparing with the scalar principal type case
we would expect the order of the operator $E\in\varPsi_\mathrm{cl}^{k-1}(X)$ to be lower.
($E$ does have the expected order when $N=1$, which is not
surprising since $P$ is of principal type then.)
Since $Q\in\varPsi_\mathrm{cl}^{k}(X)$ it follows that
if $N>1$ then the terms $\sigma_j(PE)$ in the asymptotic expansion
of the symbol of $PE$ that are homogeneous of degree $k< j \le N+k-1$ must vanish of
infinite order at $\gamma(I)$; these terms can be traced back to the operator $R$.
Even though only the principal symbol is invariantly defined a priori,
the statement has meaning in view of the symbol calculus, see~\cite[Theorem~18.1.17]{ho3}.
Since $d\lambda\ne 0$ near $\gamma(I)$, this means that
the terms in the asymptotic expansion of the symbol of $E$ that are
homogeneous of degree $k-N< \ell \le k-1$ must vanish of
infinite order at $\gamma(I)$.
(Of course, since $\gamma(I)$ has empty interior, we cannot from this
infer that $E\in\varPsi^{k-N}$ at $\gamma(I)$ in the sense
of the discussion preceding~\cite[Proposition~18.1.26]{ho3}.)
Indeed, if $w_0\in\gamma(I)$
and $\sigma_E\sim e_{k-1}+e_{k-2}+\ldots$
then the principal symbol of $PE$ is
$\sigma_{N+k-1}(PE)=\lambda^Ne_{k-1}$, so $e_{k-1}$ vanishes of infinite order at $w_0$
by Lemma \ref{appthm19special} in the appendix. If $k< j \le N+k-1$
then the only term in $\sigma_{j}(PE)$ that does not involve the functions
$e_{j-N+1},\ldots,e_{k-1}$ or their derivatives is $\lambda^Ne_{j-N}$,
so the claim follows by induction with respect to $j$ and an application of
Lemma \ref{appthm19special}.
This means that if $q$ is the principal symbol of $Q$ then
\begin{equation}\label{eq:essentialtoporder}
\partial_x^\alpha\partial_\xi^\beta(q(x,\xi)-\lambda^N(x,\xi)e_{k-N}(x,\xi))|_{(x,\xi)\in\gamma(I)}=0
\quad\text{for all }\alpha,\beta\in\N^n,
\end{equation}
since $\sigma_k(Q-PE)$ vanishes of infinite order, and
the only term in $\sigma_{k}(PE)$ that does not involve the functions
$e_{k-N+1},\ldots,e_{k-1}$ or their derivatives is $\lambda^Ne_{k-N}$.

Note also that under the hypotheses of Theorem \ref{thm:application} it
follows that $P$ is not solvable at the cone generated by $\gamma(I)$.
In the case when condition (a) holds, this is an immediate consequence
of~\cite[Theorem~2.9]{de2} in view of the discussion following
Theorem \ref{thm:solv}. If instead (b) holds, then $\mathscr P$
fails to be solvable at the cone generated by $\gamma(I)$ by
an application of Theorem \ref{thm:solv}. If $P$ is solvable
there, then the arguments in the proof of~\cite[Theorem~2.9]{de2}
can be used to arrive at a contradiction. That is, given ${}^t(f_1,\ldots,f_N)$
we set $u_1=0$, $u_2=-f_1$ and recursively $u_{j+1}=Lu_j-f_j$ for $0\le j<N$.
Then $\mathscr P \ {}^t(u_1,\ldots,u_{N-1},u_N)=(f_1,\ldots,f_{N-1},f)$,
with
\[
f=Lu_N+\sum_{j=0}^{N-1}A_ju_{j+1}=-\sum_{\ell=1}^{N-1}L^{N-\ell}f_\ell
-\sum_{j=0}^{N-1}\sum_{\ell=1}^j A_j L^{j-\ell}f_\ell.
\]
If ${}^t(f_1,\ldots,f_N)$ belongs to an appropriate local Sobolev space
determined by the definition of solvability for $P$ and the formula above,
then there is a distribution $u\in\De (X)$ such that $Pu-f-f_N$ has no
wave front set in the cone generated by $\gamma(I)$. If we put
$v_1=u$ and recursively $v_{j+1}=Lv_j$ for $1\le j<N$ then
$U={}^t(u_1,\ldots,u_N)+{}^t(v_1,\ldots,v_N)$ satisfies
$\mathscr P U={}^t(f_1,\ldots,f_N)+G$, where the wave front set of the
vector $G$ does not meet the cone generated by $\gamma(I)$, which is a contradiction.

If $P\in\varPsi_\mathrm{cl}^m(X)$ is a scalar operator
we shall, for the rest of this section only, let $\ran P$
denote the range of $P$ viewed as an operator $P:\De(X)\to\De(X)/C^\infty(X)$,
\[
\ran P=\{f\in\De(X) : f-Pu\in C^\infty(X)\text{ for some } u\in\De(X)\}.
\]
The operators $L^j$ that appear in Theorem \ref{thm:application}
enjoy the following property.

\begin{cor}\label{cor:application}
Let $L\in\varPsi_{\mathrm{cl}}^1(X)$ be a properly supported scalar operator,
and assume that the hypotheses of Theorem \ref{thm:application} hold.
Then we have the following chain of strict inclusions:
\[
\ldots \subsetneq\ran L^{k+1}\subsetneq \ran L^k\subsetneq\ldots\subsetneq \ran L\subsetneq \ran\I.
\]
In particular, if $j$ and $k$ are non-negative integers, then $\ran L^j\subset\ran L^k$ if and only if $j\ge k$.
\end{cor}

\begin{proof}
Let $k\ge 0$. If $f\in\ran L^{k+1}$, let $u\in\De(X)$
satisfy $f-L^{k+1}u\in C^\infty$.
Since $L$ is continuous $L:\De(X)\to\De(X)$,
we have $v=Lu\in\De(X)$. Now $f-L^kv=f-L^{k+1} u\in C^\infty$,
so $f\in\ran L^k$.

Conversely, assume to reach a contradiction that $\ran L^k\subset\ran L^{k+1}$, and let $K$ be
the cone given by Theorem \ref{thm:application} containing
a minimal bicharacteristic $\gamma(I)$ of the principal symbol $\lambda(w)$
of $L$. It is clear that if $\ran L^k\subset\ran L^{k+1}$ then the range of $L^k$ is microlocally contained
in the range of $L^{k+1}$ at $K$, so by an application of the theorem
with $P=L^{k+1}$ and $Q=L^k$ we obtain an operator
$E\in\varPsi_{\mathrm{cl}}^{k-1}(X)$ with symbol $e\sim e_{k-1}+e_{k-2}+\ldots$
such that, in particular, the term $\sigma_k(Q-PE)$ in the asymptotic expansion
of the symbol of $Q-PE$ that is homogeneous of degree $k$ vanishes of
infinite order at $w_0\in\gamma(I)$. Since $\lambda(w)$ is assumed
to be of principal type near $K$ there is a tangent vector
$\partial_\nu\in T_{w_0} (T^\ast(X)\smallsetminus 0)$ such that
$\partial_\nu\lambda(w_0)=\langle \partial_\nu,d\lambda\rangle\ne 0$.
In view of \eqref{eq:essentialtoporder} this implies that
\[
0=\partial_\nu^{k}(\lambda^k-\lambda^{k+1}e_{-1})|_{w_0}= k! (\partial_\nu\lambda(w_0))^k\ne 0,
\]
a contradiction. If $k=0$ this is to be interpreted as
$0=1-(\lambda(w_0))e_{-1}(w_0)=1$, which also gives a contradiction.
\end{proof}

Of course, we already know that
$\ran L^j\subset\ran L$ implies that $j>0$ under the
hypotheses of Theorem \ref{thm:application}. Indeed,
in view of Definition \ref{dfn:minimal1}
it follows by~\cite[Theorem~26.4.7$'$]{ho4} together
with~\cite[Proposition~26.4.4]{ho4} that $L$
fails to be solvable at the cone generated by $\gamma(I)$.
In view of the discussion following Definition \ref{defrangesystems},
the range of the identity is therefore not microlocally contained
in the range of $L$ at this cone, which shows that the inclusion
$\ran \I\subset\ran L$ cannot hold.

\section{Preparation}
\label{preparation}

The purpose of this section is to prove a preparation result that will be used
when proving Theorem \ref{thm:bigsystems}. We first discuss
when the kernel of a matrix valued function is a complex vector
bundle.

Let $X$ be a $C^\infty$ manifold and $P(w)$ an $N\times N$ system varying smoothly with $w\in X$,
and suppose that there is a unique section of eigenvalues
$\lambda(w)$ of $P(w)$ vanishing
along a compact and smooth simple curve
$\gamma\subset\varSigma(P)$,
where $\lambda(w)$ has constant multiplicity $J$ in a neighborhood.
Since the eigenvalues of $P(w)$ depend continuously on $w\in X$, it follows that
there exists a neighborhood $Y$ of $\gamma$ and a small constant $c>0$ such that the operator
valued function
\begin{equation*}
w\mapsto\varPi(w)=\frac{1}{2\pi i}\int_{|z|=c}(z\mathrm{Id}_N-P(w))^{-1} \, dz \in C^\infty(Y)
\end{equation*}
is the projection onto the generalized eigenvectors for the eigenvalue $\lambda(w)$ of $P(w)$
(see for example~\cite[pp.~40--45]{kato}). The dimension of the algebraic eigenspace
$\ran\varPi(w)$ equals the algebraic multiplicity of the eigenvalue $\lambda(w)$.
(We could of course use the existence of the projection to give an alternative proof
of Corollary \ref{cor:algmult}.)
Assuming also that $\dim\kernel (P(w)-\lambda(w)\mathrm{Id}_N)\equiv J$ in $Y$
it follows that $\ran\varPi(w)=\kernel (P(w)-\lambda(w)\mathrm{Id}_N)$.
Note that if $w$ is fixed then the operator $\varPi(w)$ is idempotent and
we have the direct sum
\begin{equation}\label{directsum}
\C^N=\ran\varPi(w)\oplus\ran(\mathrm{Id}_N-\varPi(w)).
\end{equation}

Let $V$ be the topological manifold
\[
V=\{(w,z):w\in Y,\ z\in\kernel(P(w)-\lambda(w)\mathrm{Id}_N)\},
\]
and let $\pi:(w,z)\mapsto w$ be the projection. Then $V$ can by means of $\varPi$
be given the structure of a $C^\infty$
complex vector bundle over $Y$. Indeed, it is clear that each fiber
$V_w=\pi^{-1}(w)=\kernel(P(w)-\lambda(w)\mathrm{Id}_N)$ over $w$
has a natural vector space structure induced from the one on $\C^N$.
Since $Y$ is open and $w\mapsto \varPi(w)$ is smooth, we can for each
$w_\alpha\in Y$ find a neighborhood $U_\alpha\subset Y$
of $w_\alpha$ such that
\begin{equation*}
w\in U_\alpha \quad
\Longrightarrow \quad \|\varPi(w)-\varPi(w_\alpha)\|<1.
\end{equation*}
Choose orthonormal
bases $\{e_{\alpha,1},\ldots,e_{\alpha,J}\}$ and $\{e_{\alpha,J+1},\ldots,e_{\alpha,N}\}$ of
$\ran\varPi(w_\alpha)$
and $\kernel\varPi(w_\alpha)=\ran(\mathrm{Id}_N-\varPi(w_\alpha))$, respectively,
so that $\varPi(w_\alpha)e_{\alpha,k}=e_{\alpha,k}$ for $1\le k\le J$ and $0$ otherwise.
It is then easy to see that the $C^\infty$ sections
\begin{equation}\label{localframe}
U_\alpha\ni w\mapsto f_{\alpha,k}(w)=\varPi(w)e_{\alpha,k},\quad 1\le k\le J,
\end{equation}
are linearly independent and therefore constitute a basis for each fiber $V_w$ over
$U_\alpha$.
(Note that the $C^\infty$ sections
\[
U_\alpha\ni w\mapsto f_{\alpha,k}(w)=(\mathrm{Id}_N-\varPi(w))e_{\alpha,k},\quad J+1\le k\le N,
\]
are also linearly independent.)
This allows for the construction of the required
local isomorphism $\psi_\alpha$ from $\pi^{-1}(U_\alpha)$ onto $U_\alpha\times\C^J$.
Hence $V$ is a $C^\infty$ complex vector bundle of fiber dimension $J$,
and \eqref{localframe} is a local frame for $V$
over $U_\alpha$.
In fact, if $w\in U_\alpha\cap U_\beta$, then
the columns of the transition matrix
\[
g_{\alpha\beta}=\psi_\alpha\circ\psi_\beta^{-1}:U_\alpha\cap U_\beta\to GL(N,\C)
\]
are just the coordinates of the local frame over $U_\beta$ in terms of the
local frame over $U_\alpha$. Since the local frames consist of $C^\infty$
sections, this implies that $g_{\alpha\beta}\in C^\infty(U_\alpha\cap U_\beta)$.
The same arguments show that
the complimentary manifold
\[
V'=\{(w,z): w\in Y,\ z\in\kernel\varPi(w)\}
\]
is a $C^\infty$ complex vector bundle over $Y$
with fiber dimension $N-J$.

We shall need the fact that $V$ and $V'$ are 
trivial
in the cases under consideration in Section \ref{sec:mainresultsforsystems}.

\begin{prop}\label{prop:basis}
Let $P\in\varPsi_\mathrm{cl}^m(X)$ be an $N\times N$ system
with homogeneous principal symbol $P_m$.
Let $\gamma$ be a compact and smooth simple curve contained
in the characteristic set $\varSigma(P_m)$ of $P_m$,
and suppose that $P$ is of principal type with constant characteristics near $\gamma$.
Let $w\mapsto\lambda(w)$ be the unique section of eigenvalues of $P_m(w)$
vanishing along $\gamma$, and suppose that $\gamma$ is a
one dimensional bicharacteristic of $\lambda$
with injective regular projection in $S^\ast(X)$.
Then there exists a conic neighborhood $\varOmega$ of $\gamma$
and a positive number $J$ such that
\begin{equation}\label{defvectorbundleV}
V=\{(w,z): w\in\varOmega,\ z\in\kernel (P_m(w)-\lambda(w)\mathrm{Id}_N)\}
\end{equation}
is a $C^\infty$ complex vector bundle over $\varOmega$
with fiber dimension $J$,
where the fiber $V_w$ over $w\in\varOmega$ is given by $V_w=\kernel (P_m(w)-\lambda(w)\mathrm Id_N)$.
Moreover, there is a local frame $\{z_1,\ldots,z_J\}$ for $V$ over $\varOmega$ such that
\[
z_k:\varOmega\ni w\mapsto z_k(w)\in V_w,\quad 1\le k\le J,
\]
is homogeneous of degree $0$ and an eigenvector of
$P_m$ with eigenvalue $\lambda$. Thus $V$ is trivial.
This local frame can be completed to a local frame for the
the trivial vector bundle $F=\varOmega\times\C^N$.
\end{prop}

\begin{proof}
By assumption $P_m$ has constant characteristics, so the characteristic equation
\[
|P_m(w)-\lambda\mathrm{Id}_N|=0
\]
has the unique local solution $\lambda(w)\in C^\infty$ of multiplicity $J>0$,
where $\lambda(w)$ is the section of eigenvalues given in the statement of
the proposition. Since $P_m$ is of principal type, the geometric
multiplicity $\dim\kernel (P_m(w)-\lambda(w)\mathrm{Id}_N)\equiv J$
in a neighborhood of $\gamma$ by~\cite[Proposition~2.10]{de2}.
If $\pi:T^\ast(X)\smallsetminus 0\to S^\ast(X)$ is the projection,
it follows by homogeneity that we can find a neighborhood $\mathcal V\subset S^\ast(X)$ of
$\pi\circ\gamma$ such that
this still holds in
the conic set $\pi^{-1}(\mathcal V)\subset T^\ast(X)\smallsetminus 0$.

By introducing a Riemannian metric on $X$
defining the unit cotangent bundle,
we can as in the proof of Proposition \ref{prop:homogeneouseigenvalues}
write
$P_m(x,\xi)=|\xi|^m\pi^\ast p_s(x,\xi)$ and $\lambda(x,\xi)=|\xi|^m\pi^\ast\varrho_s(x,\xi)$
where $p_s$ and $\varrho_s$ are functions in $C^\infty(S^\ast(X))$ with values in
$\mathcal L_N$ and $\C$, respectively.
In the neighborhood $\mathcal V$ of $\pi\circ\gamma$ it follows by homogeneity
that $\varrho_s$ is the unique section of eigenvalues of $p_s$ that vanishes along $\pi\circ\gamma$.
In particular, $\pi\circ\gamma\subset\varSigma(p_s)$.
With $v=\pi(w)$ for $w=(x,\xi)\in \pi^{-1}(\mathcal V)$ it is also easy to see
that
\[
\kernel (P_m(w)-\lambda(w)\mathrm{Id}_N)=\kernel (p_s(v)-\varrho_s(v)\mathrm{Id}_N).
\]
Thus, $\dim\kernel(p_s(v)-\varrho_s(v)\mathrm{Id}_N)\equiv J$ for $v\in\mathcal V$. By the
discussion preceding the proposition it then follows
that
\[
\{(v,z): v\in\mathcal V,\ z\in\kernel(p_s(v)-\varrho_s(v)\mathrm{Id}_N)\}
\]
is a $C^\infty$ complex vector bundle over $\mathcal V$.
With the notation
of the proposition it is clear that the pullback by
$\pi$ of the local frames constructed above yield local frames for
$V$ over open conic subsets of $T^\ast(X)\smallsetminus 0$
whose union forms a conic neighborhood of $\gamma$,
which proves the first part of the proposition.
Similarly, if we can find $C^\infty$ sections $v\mapsto z_k(v)$, $1\le k\le J$,
constituting a basis for $\kernel (p_s(v)-\varrho_s(v)\mathrm{Id}_N)$
for every $v\in\mathcal V$, then the collection $\{\pi^\ast z_k\}_{k=1}^J$ is a
local frame for $V$ over $\pi^{-1}(\mathcal V)$
and the $C^\infty$ sections $\pi^\ast z_k$ are
homogeneous of degree $0$.
In particular, if $p_s(v)z_j(v)=\varrho_s(v)z_j(v)$ for $v=\pi(w)\in\mathcal V$ then
\[
P_m(w)\pi^\ast z_j(w)=\lambda(w)\pi^\ast z_j(w).
\]
If the local frame $\{z_1,\ldots,z_J\}$ can be extended
to a local frame for the trivial complex vector bundle
$\mathcal V\times\C^N$,
then the collection
$\{\pi^\ast z_j\}_{j=1}^N$ has the required properties,
thereby proving the proposition.
Since $\gamma$ has injective regular projection in $S^\ast(X)$,
we can thus assume that $\gamma$ is a curve on the cosphere bundle
to begin with, while $P_m(w)$ and $\lambda(w)$ belong to $C^\infty(S^\ast(X))$.
Thus \eqref{defvectorbundleV} holds with $\varOmega$ replaced by $\mathcal V$.
Since $\gamma$ is contractible by assumption, we find
(after possibly shrinking $\mathcal V$
if necessary) that $V$ is 
trivial,
see for example Corollary $4.8$ in~\cite[Chapter~3]{hu}.
If
\[
V'=\{(w,z):w\in\mathcal V,\ z\in\kernel\varPi(w)\}
\]
is the complimentary vector bundle over $\mathcal V$, the
same reasoning shows that $V'$ is 
trivial.
Since the operator valued function $w\mapsto\mathrm{Id}_N-\varPi(w)$
is the projection onto the null space $V_w'$ of $\varPi(w)$,
we have $\C^N=V_w\oplus V_w'$ for any $w\in\mathcal V$
by \eqref{directsum},
so together the local frames for $V$ and $V'$ over $\mathcal V$ give
a local frame for the trivial vector bundle $F=\mathcal V\times\C^N$ over $\mathcal V$.
This completes the proof.
\end{proof}

We now prove that
the local preparation result for systems given by Lemma $4.1$ in~\cite{de2}
can be generalized to a neighborhood of a compact one dimensional bicharacteristic interval.

\begin{lem}\label{lem:prep}
Let $P\in\varPsi_\mathrm{cl}^m(X)$ be an $N\times N$ system
with principal symbol $P_m$.
Let $\gamma$ be a compact and smooth simple curve contained
in the characteristic set $\varSigma(P_m)$ of $P_m$,
and suppose that $P$ is of principal type with constant characteristics
near $\gamma$.
Let $\lambda(w)$ be the unique section of eigenvalues of $P_m(w)$
vanishing along $\gamma$, and suppose that $\gamma$ is a
one dimensional bicharacteristic of $\lambda$
with injective regular projection in $S^\ast(X)$.
Then one can find $N\times N$ systems $A$ and
$B$ in $\varPsi_\mathrm{cl}^0(X)$, non-characteristic in
a conic neighborhood of $\gamma$, such that
\begin{equation}\label{1stnormalforsystems}
APB=\left( \begin{array}{cc} \tilde{P}_{11} & 0\\
0 & \tilde{P}_{22} \end{array} \right)\in\varPsi_{\mathrm{cl}}^m(X)
\end{equation}
microlocally near $\gamma$. Moreover, $\tilde{P}_{22}$ is elliptic,
and we have $\sigma(\tilde{P}_{11})=\lambda\mathrm{Id}_J$
where the section of eigenvalues $\lambda(w)\in C^\infty$ of $P(w)$
is of principal type near $\gamma$.
\end{lem}

\begin{proof}
First we note, as in the beginning of the proof of Proposition \ref{prop:basis},
that since $P_m$ is of principal type with constant characteristics, the geometric
multiplicity $\dim\kernel (P_m(w)-\lambda(w)\mathrm{Id}_N)\equiv J>0$
in a conic neighborhood of $\gamma$, where $\lambda(w)\in C^\infty$ is
the section of eigenvalues of multiplicity $J$ given in the statement of
the lemma.
Moreover,~\cite[Proposition~2.10]{de2}
also gives that $d\lambda\ne 0$ on $\gamma$, and since $\gamma$
is a one dimensional bicharacteristic of $\lambda$ with injective projection in $S^\ast(X)$ it follows that
the composition of $\gamma$ and the Hamilton vector field $H_\lambda$ of $\lambda$ does not have the radial direction.
Indeed, as in the proof of~\cite[Theorem~26.4.12]{ho4} we can
for a suitably normalized parametrization $t\mapsto\gamma(t)$ of $\gamma$ find
a $C^\infty$ function $\varrho$, homogeneous of degree $0$, such that
\begin{equation*}
0\ne \gamma'(t)=\varrho(\gamma(t)) H_\lambda(\gamma(t))=H_{\re \varrho\lambda}\circ\gamma(t).
\end{equation*}
In particular, since $\lambda\circ\gamma=0$ we find that
$\gamma$ is a bicharacteristic of the homogeneous function
$\re \varrho\lambda$ such that $H_{\re \varrho\lambda}\ne 0$ along $\gamma$.
Thus, if $H_\lambda$ and the radial vector field are
linearly dependent at some point on $\gamma$,
then $\gamma$ would just be a ray in the radial direction
which is a contradiction since $\gamma$ is assumed to have injective projection
in $S^\ast(X)$.
By homogeneity $\lambda$ is then of principal type in a conic neighborhood of $\gamma$.

Now use Proposition \ref{prop:basis} and an orthogonalization procedure to
obtain a unitary $N\times N$ system $E$, homogeneous of degree $0$
and non-characteristic in a conic neighborhood of $\gamma$, such that
\[
E^\ast P_mE=\left( \begin{array}{cc} \lambda(w)\mathrm{Id}_J & P_{12}\\
0 & P_{22} \end{array} \right)=\tilde{P}_m
\]
is the principal symbol of $A'PB'$ for any systems $A',B'\in\varPsi_{\mathrm{cl}}^0(X)$ having
principal symbols $E^\ast$ and $E$, respectively.
Here $E^\ast=E^{-1}$ is the Hermitian adjoint of $E$.
Inspecting the end of the proof of~\cite[Lemma~4.1]{de2} we find that
the result now follows by essentially repeating the arguments found there.
We omit the details.
\end{proof}

\section{The proof of Theorem \ref{thm:bigsystems}}
\label{proofrange}

Recall that we may assume that the systems $P$ and $Q$ given
by Theorem \ref{thm:bigsystems} have the same order.
Now note that if $A$ and $B$ are the elliptic systems
given by Lemma \ref{lem:prep}, 
it follows as a special case of Proposition \ref{prop.26.4.4} that
the range of $Q$ is microlocally contained in the range of $P$ at $K$
if and only if the range of
$AQB$ is microlocally contained in the range of $APB$ at $K$.
Since $A$ and $B$ are elliptic, it is easy to see using the calculus
that all the terms in the asymptotic expansion of the symbol of $Q$
have vanishing Taylor coefficients if and only if the same holds for
$AQB$. Note also that when $\gamma$ is a minimal characteristic point,
the normal form given by Lemma \ref{lem:prep}
is still valid near $\gamma$ in view of~\cite[Lemma~4.1]{de2}.
We can thus reduce the proof of Theorem \ref{thm:bigsystems} to the case when $P$ has the form given by 
\eqref{1stnormalforsystems} and $\lambda=\lambda(w)$
is the unique eigenvalue of the principal symbol $P_m$ of $P$ satisfying $\lambda\circ\gamma = 0$.
In view of Lemma \ref{sys:symbolvanishafterconjugation} in the appendix,
we can use Proposition \ref{prop.26.4.4} together with~\cite[Theorem~21.3.6]{ho3}
or~\cite[Theorem~26.4.13]{ho4}
when $\gamma$ is a characteristic point
or a one dimensional bicharacteristic, respectively,
to further reduce the proof to the case $Q, P\in\varPsi_\mathrm{cl}^1(\mathbb{R}^n)$,
$\gamma(x_1)=(x_1,0,\varepsilon_n)\in T^\ast(\mathbb{R}^n)$ for $x_1\in I$, and
\begin{equation}\label{eq:sysnormal1}
\lambda(x,\xi)=\xi_1+if(x,\xi')
\end{equation}
where $f$ is real valued, homogeneous of degree $1$ and independent of $\xi_1$.
Note that under these hypotheses, $\gamma$ is still a minimal characteristic point
or a minimal bicharacteristic interval of $\lambda$.
Thus, in any neighborhood of $\gamma$ one can find an interval in the $x_1$ direction where $f$
changes sign from $-$ to $+$
for increasing $x_1$.
If $\gamma$ is not reduced to a point then
$f$ vanishes of infinite order on $\gamma$ by Proposition \ref{prop:minimal03},
and by Theorem \ref{thm:minimal07} we can find a sequence
$\{\varGamma_j\}_{j=1}^\infty$ of $\varrho_j$-minimal
bicharacteristic intervals such that $\varrho_j\to 0$ and $\varGamma_j\to\gamma$ as
$j\to\infty$.
Note also that we still have
\begin{equation}\label{eq:1stnormalforsystems1}
P=\left( \begin{array}{cc} P_{11} & 0\\
0 & P_{22} \end{array} \right)
\end{equation}
where $\sigma(P_{11})=\lambda\mathrm{Id}_J$
and $P_{22}$ is elliptic microlocally near $\gamma(I)$.

Let 
\begin{equation}\label{1stnormalforsystems1}
Q=\left( \begin{array}{cc} Q_{11} & Q_{12}\\
Q_{21} & Q_{22} \end{array} \right)
\end{equation}
be the block form of $Q$ corresponding to \eqref{eq:1stnormalforsystems1},
so that for example $Q_{12}$ is a $J\times (N-J)$ system.
If $P_{22}^{-1}$ is a microlocal parametrix of $P_{22}$ near $\gamma(I)$, then
\begin{equation*}
Q=P\cdot\left( \begin{array}{cc} 0 & 0\\
P_{22}^{-1}Q_{21} & P_{22}^{-1}Q_{22} \end{array} \right)
+ \left( \begin{array}{cc} Q_{11} & Q_{12}\\
0 & 0 \end{array} \right)
\end{equation*}
mod $\varPsi^{-\infty}$ microlocally near $\gamma(I)$.
We may assume that $J>0$ since otherwise $P$ is elliptic.
Now let
\begin{equation}\label{psblockQ11}
\sigma_{Q_{11}}=q_1+q_0+\ldots
\end{equation}
be the total symbol of $Q_{11}$,
where $q_j$ is a $J\times J$ system, homogeneous of degree $j$. With $\lambda$
given by \eqref{eq:sysnormal1} we have $\sigma(P_{11})=0$ and $|\partial_{\xi_1}\sigma(P_{11})|\ne 0$
at $\gamma(I)$, so in place of the Malgrange preparation theorem we can use
(the transpose of)~\cite[Theorem~A.4]{de0} to obtain
\[
\sigma(Q_{11})(x,\xi)=\lambda(x,\xi)\mathrm{Id}_J E_0(x,\xi)+R_1(x,\xi')
\]
in a neighborhood of $\gamma(I)$
for some matrix valued smooth functions $E_0$ and $R_1$, where $R_1$
is independent of $\xi_1$. (Of course, since $\sigma(P_{11})=\lambda\I_J$,
the usual scalar Malgrange preparation theorem is actually sufficient.)
Restricting to $\xi=1$ and
extending by homogeneity we can make $E_0$ and $R_1$ homogeneous of degree
$0$ and $1$, respectively. We can repeat the argument for
lower order terms and obtain $Q_{11}=P_{11}\circ E_{11}+R_{11}(x,D_{x'})$
where $E_{11}\in \varPsi_\mathrm{cl}^0(\mathbb{R}^n)$ and
$R_{11}\in \varPsi_\mathrm{cl}^1(\mathbb{R}^n)$ are $J\times J$ systems,
and the symbol of $R_{11}$ is independent of $\xi_1$.
Doing the same for $Q_{12}$ we get
\begin{equation*}
Q=P\cdot\left( \begin{array}{cc} E_{11} & E_{12}\\
P_{22}^{-1}Q_{21} & P_{22}^{-1}Q_{22} \end{array} \right)
+ \left( \begin{array}{cc} R_{11} & R_{12}\\
0 & 0 \end{array} \right)
\end{equation*}
mod $\varPsi^{-\infty}$ microlocally near $\gamma(I)$. One easily checks that
the range of
\begin{equation}\label{1stnormalforsystems3}
R(x,D_{x'})=\left( \begin{array}{cc} R_{11}(x,D_{x'}) & R_{12}(x,D_{x'})\\
0 & 0 \end{array} \right)
\end{equation}
is microlocally contained in the range of $P$ near $\gamma(I)$.
Hence Theorem \ref{thm:bigsystems} follows
if we show that all the terms in the asymptotic expansion of the symbol
of $R$ have vanishing Taylor coefficients at
$\gamma=\{(x_1,0,\varepsilon_n):x_1\in I\}$.
Note that when proving this we may assume that
the lower order terms in the symbol of $P_{11}$ are independent of $\xi_1$.
In fact, if $\sigma_{P_{11}}=\lambda\I_J+p_0+\ldots$
then~\cite[Theorem~A.4]{de0} implies that
\[
p_0(x,\xi)=a(x,\xi)(\xi_1+if(x,\xi'))+b(x,\xi')
\]
where $a$ is homogeneous of degree $-1$ and $b$ homogeneous of degree $0$, as
demonstrated in the construction of the systems $E$ and $R$ above. The
term of degree $0$ in the symbol of $(\I_J-a(x,D))P_{11}$ is equal to $b(x,\xi')$.
Repetition of the argument implies that
there exists a $J\times J$ system of classical operators $A_{11}(x,D)$ of order $-1$ such that
$(\I_J-A_{11})P_{11}$ has principal symbol
$(\xi_1+if(x,\xi'))\I_J$ and all lower order terms are independent of $\xi_1$.
If $A$ is the $N\times N$ system
\[
A(x,D)=\left( \begin{array}{cc} \I_J-A_{11}(x,D) & 0\\
0 & \I_{N-J} \end{array} \right)
\]
then the microlocal property of pseudodifferential
operators immediately implies that the range of $AQ$ is microlocally contained in the range
of $AP$ at $K$.
Hence, if there are systems $E$ and $R$ with
\[
R=AQ-AP E
\]
such that all terms in the asymptotic expansion of the symbol
of $R$ have vanishing
Taylor coefficients at $\gamma(I)$, then this also holds for
the symbol of $Q-PE\equiv A^{-1}R$ mod $\varPsi^{-\infty}$,
since the calculus gives that this property is preserved under composition with elliptic
systems.

When $\gamma$ is a minimal bicharacteristic interval,
that is, when $I$ is not reduced to a point, then we may assume
that there exists a neighborhood of $\gamma$ where the implication
\begin{equation}\label{eq:sysmainthm1}
f(x,\xi')=0 \quad \Longrightarrow \quad \partial f (x,\xi') / \partial x_1 \leq 0
\end{equation}
holds. Indeed,
if there is no such neighborhood, then
as shown in~\cite{jw} (see the discussion in connection with equation $(2.19)$ there),
we find that $\gamma$ is just a point and
there exists a point sequence $\{\gamma_j\}_{j=1}^\infty=\{(t_j,x_j',0,\xi_j')\}_{j=1}^\infty$ such that
$\gamma_j\to \gamma$ as $j\to\infty$, and
\begin{equation}\label{eq:sysrmk1}
f(t_j,x_j',\xi_j')=0, \quad \partial f(t_j,x_j',\xi_j') / \partial x_1 >0
\end{equation}
for each $j$.
Note that \eqref{eq:sysrmk1} implies $\{\re \lambda, \im \lambda\}(\gamma_j)>0$ and $\lambda(\gamma_j)=0$
for each $j$ since $\gamma_j=(t_j,x_j',0,\xi_j')$.
Thus, when $\gamma$ is a minimal characteristic point we
conclude that either there is a neighborhood where \eqref{eq:sysmainthm1} holds,
or we can find a sequence $\{\gamma_j\}_{j=1}^\infty$ with the properties given above.
This will allow us to complete the proof of Theorem \ref{thm:bigsystems}
using the following two results.

\begin{thm}\label{thm:sysmainthm1}
Let the $N\times N$ system $P$ be given by
\eqref{eq:1stnormalforsystems1} where $P_{22}$ is elliptic, and
suppose that in a conic neighborhood $\varOmega$ of
\[
\varGamma'=\{(x_1,x',0,\xi'), \ a\leq x_1 \leq b\}\subset T^*(\mathbb{R}^n)\smallsetminus 0
\]
the principal symbol of $P_{11}$ has the form $\lambda(x,\xi)\mathrm{Id}_J$ with
\[
\lambda(x,\xi)=\xi_1+if(x,\xi'),
\]
where $f$ is real valued and homogeneous of degree $1$,
while the lower order terms of the symbol of $P_{11}$
are all independent of $\xi_1$.
Suppose also that
\eqref{eq:sysmainthm1} holds
in $\varOmega$ and that in
any neighborhood of $\varGamma'$ one can find an interval in the $x_1$ direction where $f$
changes sign from $-$ to $+$
for increasing $x_1$.
Assume that if $b>a$
then $f$ vanishes of infinite order on $\varGamma'$ and there exists a $\varrho\ge 0$ such that
for any $\varepsilon>\varrho$ one can
find a neighborhood of
\begin{equation*}
\varGamma_\varepsilon'=\{(x_1,x',0,\xi'), \ a+\varepsilon \leq x_1 \leq b-\varepsilon\}
\end{equation*}
where $f$ vanishes identically.
Furthermore, let the $N\times N$ system $R(x,D_{x'})$
be given by \eqref{1stnormalforsystems3}
and suppose that in $\varOmega$
the symbol of $R$ is given by an asymptotic sum
of homogeneous terms that are all independent of $\xi_1$.
If there exists a compactly based cone $K\subset T^*(\mathbb{R}^n)\smallsetminus 0$
containing $\varOmega$
such that the range of $R$ is microlocally contained in the range of $P$ at $K$, 
then all the terms in the asymptotic sum of the symbol of $R$ have vanishing Taylor coefficients on
$\varGamma_\varrho'$ if $a<b$, and at $\varGamma'$ if $a=b$.
\end{thm}

\begin{thm}\label{thm:sysmainthm2}
Let the $N\times N$ system $P$ be given by
\eqref{eq:1stnormalforsystems1} where $P_{22}$ is elliptic, and
suppose that in a conic neighborhood $\varOmega$ of
\[
\varGamma'=\{(0,\varepsilon_n)\}\subset T^*(\mathbb{R}^n)\smallsetminus 0
\]
$P_{11}$ has the form $P_{11}=(D_1+ix_1D_n)\mathrm{Id}_J$.
Moreover, let
the $N\times N$ system $R(x,D_{x'})$
be given by \eqref{1stnormalforsystems3}
and suppose
that the symbol of $R$ is given by an asymptotic sum
of homogeneous terms that are all independent of $\xi_1$.
If there exists a compactly based cone $K\subset T^*(\mathbb{R}^n)\smallsetminus 0$
containing $\varOmega$
such that the range of $R$ is microlocally contained in the range of $P$ at $K$, 
then all the terms in the asymptotic sum of the symbol of $R$ have vanishing
Taylor coefficients on $\varGamma'$.
\end{thm}

Postponing the proofs of these results
we are now left with three cases:

i) $\gamma$ is a minimal bicharacteristic interval. Then there is a neighborhood
$\varOmega$ of $\gamma$ where \eqref{eq:sysmainthm1}
holds, and since $\im\lambda=f$ is homogeneous we may assume that $\varOmega$ is conic.
By Theorem \ref{thm:minimal07} there exists a sequence $\{\varGamma_j\}_{j=1}^\infty$ of $\varrho_j$-minimal
bicharacteristic intervals such that $\varrho_j\to 0$ and $\varGamma_j\to\gamma$ as
$j\to\infty$. For sufficiently large $j$ we have $\varGamma_j\subset\varOmega$.
If
\[
\varGamma_j=\{(x_1,x_j',0,\xi_j'): a_j\le x_1\le b_j\}
\]
then all the terms in the asymptotic sum of the symbol of $R$
vanish of infinite order on
\[
\varGamma_{\varrho_j}=\{(x_1,x_j',0,\xi_j'): a_j+\varrho_j\le x_1\le b_j-\varrho_j\}
\]
by Theorem \ref{thm:sysmainthm1}. Since $\varGamma_{\varrho_j}\to\gamma$
as $j\to\infty$, and all the terms in the asymptotic sum of the symbol
of $R$ are smooth functions, it follows that all the terms in the
asymptotic sum of the symbol of $R$ vanish of infinite order on $\gamma$,
thus proving Theorem \ref{thm:bigsystems} in this case.

ii) $\gamma$ is a minimal characteristic point and \eqref{eq:sysmainthm1}
holds. Then all the terms in the
asymptotic sum of the symbol of $R$ vanish of infinite order on $\gamma$
by Theorem \ref{thm:sysmainthm1},
so Theorem \ref{thm:bigsystems} follows.

iii) $\gamma$ is a minimal characteristic point and \eqref{eq:sysmainthm1}
is false. Let $\gamma_j$ be a fixed point in the sequence $\{\gamma_j\}_{j=1}^\infty$
satisfying \eqref{eq:sysrmk1}. Since
$P$ is given by \eqref{eq:1stnormalforsystems1} and
the principal symbol of $P_{11}$ is just the scalar function $\lambda$ times the identity matrix,
we can then by conjugating as in the scalar case (see the proof of~\cite[Theorem~26.3.1]{ho4})
show that $P_{11}$ is microlocally conjugate to $(D_1+ix_1D_n)\mathrm{Id}_J$,
which allows us to prove Theorem \ref{thm:bigsystems} by an application of
Theorem \ref{thm:sysmainthm2}.
We prove this by adapting the arguments in~\cite[p.~18]{de1}, where it is shown to hold for
systems of semiclassical operators.
Note that we now forgo the previous preparation $Q=PE+R$ with $R$
given by \eqref{1stnormalforsystems3},
with the intention of recreating it after having
conjugated $P$.

Since $\gamma_j$ is fixed, we can
by choosing appropriate local coordinates
use~\cite[Theorem~21.3.3]{ho3}
to find a canonical
transformation $\chi$ and a smooth function $\mu$
such that $\chi(0,\varepsilon_n)=\gamma_j$
and $\chi^\ast(\mu\lambda)=\xi_1+ix_1\xi_n$ near $(0,\varepsilon_n)$.
By~\cite[Theorem~26.3.1]{ho4} together with
Lemma \ref{sys:symbolvanishafterconjugation} in the appendix
we can then find systems $\tilde A$ and $\tilde B$ of Fourier integral operators
such that $\tilde P=\tilde B P\tilde A$ is still on a normal form of the type \eqref{eq:1stnormalforsystems1}
with $\sigma(\tilde P_{11})=(\xi_1+ix_1\xi_n)\mathrm{Id}_J$
in a conic neighborhood $\varOmega$ of $(0,\varepsilon_n)$ and $\tilde P_{22}$ elliptic.
Let therefore
\[
\tilde P_{11}=\lambda(x,D)\mathrm{Id}_J+F,
\]
where $\lambda(x,\xi)=\xi_1+ix_1\xi_n$ and $F\in\varPsi_\mathrm{cl}^0(\R^n)$
has a symbol with asymptotic expansion $\sigma_F(w)\sim \sum_{j\ge 0} F_{-j}(w)$.
Here $F_{-k}$ is a matrix valued function, homogeneous of degree $-k$.
Let the systems $A,B\in\varPsi_\mathrm{cl}^0(\R^n)$ have symbols
$\sigma_A\sim\sum_{j\ge 0} A_{-j}$ and $\sigma_B\sim\sum_{j\ge 0} B_{-j}$
with $A_0(w)\equiv B_0(w)$. Then the calculus gives
\[
\tilde P_{11}A-B\lambda(x,D)\mathrm{Id}_J=E\in\varPsi_\mathrm{cl}^0(\R^n),
\]
where the system $E$ has symbol $\sigma_E\sim\sum_{j\ge0}E_{-j}$ and
\[
E_{-k}=\lambda (A_{-k-1}-B_{-k-1})+F_0 A_{-k}
+\partial_\xi \lambda D_x A_{-k} - \partial_\xi B_{-k} D_x \lambda+R_{-k}.
\]
Here $R_{-k}$ only depends on $A_{-j},B_{-j}$ for $j<k$ and $R_0\equiv 0$. Using the fact that
\begin{equation*}
\begin{aligned}
\partial_\xi \lambda D_x A_{-k} & - \partial_\xi B_{-k} D_x \lambda =
\frac{1}{2i} H_\lambda  (A_{-k} + B_{-k} )\\
& + \frac{1}{2i}
\Big( (\partial_x \lambda ) \partial_\xi ( A_{-k} - B_{-k} ) +  (\partial_\xi \lambda ) \partial_x ( A_{-k} - B_{-k} ) \Big),
\end{aligned}
\end{equation*}
where $H_\lambda$ is the Hamilton vector field of $\lambda$,
we can therefore write
\[
E_{-k}=\frac{1}{2i} H_\lambda  (A_{-k} + B_{-k} )+ \lambda (A_{-k-1}-B_{-k-1})+F_0 A_{-k}
+R_{-k},
\]
where $R_{-k}$ now also depends on the difference $A_{-k} -B_{-k}$
in addition to $A_{-j},B_{-j}$ for $j<k$.
Note that since $A_0(w)\equiv B_0(w)$ we still have $R_0\equiv 0$.
Now we can choose $A_0$ so that $A_0=\mathrm{Id}_J$ on
$V_0=\{w:\im\lambda(w)=0\}$ and $\frac{1}{i}H_\lambda A_0+F_0A_0$
vanishes of infinite order on $V_0$ near $(0,\varepsilon_n)$.
In fact, since $\{\re\lambda,\im\lambda\}\ne 0$ at $(0,\varepsilon_n)$,
we find that $H_{\re\lambda}$ and $H_{\im\lambda}$
are linearly independent at $(0,\varepsilon_n)$,
and that $H_{\re\lambda}$ is not tangent to $V_0$ at
$(0,\varepsilon_n)$. In particular, $V_0$ is near
$(0,\varepsilon_n)$ a hypersurface such that 
$H_{\re\lambda}$ is transversal to $V_0$. Thus, the equation
determines all derivatives of $A_0$ on $V_0$,
and we can use Borel's theorem to obtain a solution.
Next, we set
\[
B_{-1}-A_{-1}=\Big( \frac{1}{i}H_\lambda A_0+F_0A_0\Big)\lambda^{-1}\in C^\infty
\]
and obtain $E_0\equiv 0$. This also completely determines
$R_{-1}$. Similarly, lower order terms are eliminated by making
\[
\frac{1}{2i} H_\lambda  (A_{-k} + B_{-k} )+F_0 A_{-k}
+R_{-k}
\]
vanish of infinite order on $V_0$. Note that since only the difference $B_{-k}-A_{-k}$ was determined in
the previous step, this equation can be solved for $A_{-k}$, which then also determines $B_{-k}$.
Next, by choosing $B_{-k-1}-A_{-k-1}$ appropriately we obtain $E_{-k}\equiv 0$, and in the process
we also completely determine $R_{-k-1}$.
Since $B$ is microlocally invertible near
$(0,\varepsilon_n)$
by construction,
we find that
\[
B^{-1}\tilde P_{11}A\equiv \lambda(x,D)\mathrm{Id}_J\quad\text{mod }\varPsi^{-\infty}
\]
near
$(0,\varepsilon_n)$,
if $B^{-1}$ is a properly supported
microlocal parametrix of $B$. Since Definition \ref{defrangesystems}
is invariant under this type of composition by the discussion
in the first paragraph of this section, we can
let $A$ and $B^{-1}$ be included in the systems $\tilde A$ and
$\tilde B$ of Fourier integral operators already introduced,
and repeat the arguments above to obtain
\begin{equation*}
\tilde BQ\tilde A =\tilde BP\tilde A E+R(x,D_{x'}),
\end{equation*}
where $R$ is of the form \eqref{1stnormalforsystems3}
in a neighborhood of $(0,\varepsilon_n)$
with range microlocally contained in the range of $\tilde BP\tilde A$ at
some compactly based cone $K'$ containing $\varOmega$, and $E$ and $R$ have classical symbols.
Then all the terms
in the asymptotic expansion of the symbol of
$R$ vanish of infinite order at $(0,\varepsilon_n)$
by Theorem \ref{thm:sysmainthm2}. If appropriate
systems $\tilde A'$ and $\tilde B'$ of Fourier integral
operators are chosen as in the proof of
Lemma \ref{sys:symbolvanishafterconjugation} in the appendix,
that is,
\[
W\! F(\tilde A' \tilde B-\I_N)\cap K=\emptyset,\quad
W\! F(\tilde A \tilde B '-\I_N)\cap K=\emptyset,\quad
\]
then Lemma \ref{sys:symbolvanishafterconjugation}
implies that all the terms in the asymptotic expansion of the symbol of
\[
Q-P\tilde A E \tilde B'\equiv
\tilde A' (\tilde BQ\tilde A  - \tilde BP\tilde A E )\tilde B'
= \tilde A' R(x,D_{x'}) \tilde B'
\quad\text{mod }\varPsi^{-\infty}(K)
\]
vanish of infinite order at $\gamma_j$.
Note that $\tilde A' R(x,D_{x'}) \tilde B'$ now has the block form
\eqref{1stnormalforsystems3} in a neighborhood of $\gamma_j$. However, this
neighborhood does not necessarily contain $\gamma$ and the symbol
is no longer necessarily independent of $\xi_1$.

We have now shown that for each $j$ there exists an operator
$E_j\in\varPsi_{\mathrm{cl}}^{0}(\mathbb{R}^n)$ such that
all the terms
in the asymptotic expansion of the symbol of $Q-PE_j$
have vanishing Taylor coefficients at $\gamma_j$.
To construct the operator $E$ in Theorem \ref{thm:bigsystems}, we do the following.
For each $j$, write $E_j$ in block form corresponding to
that of $P$ as $E_j=(E_{k\ell,j})$, $k,\ell=1,2$, and for $k=\ell=1$ denote the symbol of $E_{11,j}$ by
\[
e^j(x,\xi)\sim\sum_{\ell=0}^\infty e_{-\ell}^j(x,\xi)
\]
where $e_{0}^j(x,\xi)$ is the principal part,
and $e_{-\ell}^j(x,\xi)$ is homogeneous of degree $-\ell$.
With $Q$ given by \eqref{1stnormalforsystems1}
and the symbol of $Q_{11}$ given by
\eqref{psblockQ11}, let $\sigma_{P_{11}}=p_1+p_0+\ldots$ so that
$p_1=\lambda\I_J$ is the principal symbol
of $P_{11}$. It then follows by
Proposition \ref{appthm26} in the appendix that
there exists a matrix valued function $e_0\in C^\infty(T^\ast(\mathbb{R}^n)\smallsetminus 0,\mathcal L_J)$,
homogeneous of degree $0$, such that $q_1-p_1 e_0$ has vanishing Taylor coefficients at
$\gamma$.

This argument can be repeated for lower order terms. Indeed, the term of degree $0$
in the symbol of $Q_{11}-P_{11}E_{11,j}$ is
\[
\sigma_0(Q_{11}-P_{11}E_{11,j})=\tilde{q}_j-p_1 e_{-1}^j,
\]
where
\begin{align*}
\tilde{q}_j(x,\xi)&=q_0(x,\xi)
-p_0(x,\xi) e_0^j(x,\xi)
-\sum_k \partial_{\xi_k} p_1(x,\xi) D_{x_k}e_0^j(x,\xi) .
\end{align*}
We can write
\[
p_1(x,\xi)e_{-1}^j(x,\xi)=p_1(x,\xi/|\xi|)e_{-1}^j(x,\xi/|\xi|),
\]
so that $\tilde{q}_j(x,\xi)$, $p_1(x,\xi/|\xi|)$ and $e_{-1}^j(x,\xi/|\xi|)$ are
all homogeneous of degree $0$. Since
\[
\partial_x^\alpha\partial_\xi^\beta e_0(\gamma)=\lim_{j\to\infty}
\partial_x^\alpha\partial_\xi^\beta e_0^j(\gamma_j)
\]
it follows by Proposition \ref{appthm26} in the appendix
that there is a matrix valued function $g\in C^\infty(T^\ast(\mathbb{R}^n)\smallsetminus 0,\mathcal L_J)$,
homogeneous of degree $0$, such that
\[
q_0(x,\xi)
-p_0(x,\xi) e_0(x,\xi)
-\sum_k \partial_{\xi_k} p_1(x,\xi) D_{x_k} e_0(x,\xi)
- p_1(x,\xi/|\xi|)g(x,\xi)
\]
has vanishing Taylor coefficients at $\gamma$. Putting $e_{-1}(x,\xi)=|\xi|^{-1}g(x,\xi)$
we find that
\[
\partial_x^\alpha\partial_\xi^\beta e_{-1}(\gamma)=\lim_{j\to\infty}
\partial_x^\alpha\partial_\xi^\beta e_{-1}^j(\gamma_j),
\]
and that
\begin{align*}
\sigma_0(Q_{11}-P_{11}\circ e_0(x,D)-P_{11}\circ e_{-1}(x,D))
\end{align*}
has vanishing Taylor coefficients at $\gamma$.
Continuing this way we successively
obtain matrix valued functions $e_m(x,\xi)\in
C^\infty(T^\ast(\mathbb{R}^n)\smallsetminus 0,\mathcal L_J)$,
homogeneous of degree $m$ for $m\le 0$, such that
\[
\sigma_{Q_{11}}-(\sum_{m=0}^M e_{-m})\sigma_{P_{11}} \quad \text{mod }S_{\mathrm{cl}}^{-M}
\]
has vanishing Taylor coefficients at $\gamma$. If we let $E_{11}$ have symbol
\[
\sigma_{E_{11}}(x,\xi)\sim \sum_{m=0}^\infty (1-\phi(\xi))e_{-m}(x,\xi)
\]
with scalar $\phi\in C_0^\infty$ equal to $1$ for $\xi$ close to $0$,
then $E_{11}\in\varPsi_{\mathrm{cl}}^{0}(\mathbb{R}^n)$ and all terms in the
asymptotic expansion of the symbol of $Q_{11}-P_{11}E_{11}$ have vanishing
Taylor coefficients at $\gamma$.
Given that $P$ has the form \eqref{eq:1stnormalforsystems1},
these arguments can be repeated to construct a $J\times (N-J)$ system $E_{12}$
such that all terms in the
asymptotic expansion of the symbol of $Q_{12}-P_{11}E_{12}$ have vanishing
Taylor coefficients at $\gamma$. By substituting Proposition \ref{appthm26elliptic}
for Proposition \ref{appthm26} throughout, these arguments also show that
there is an $(N-J)\times J$ system $E_{21}$ and an
$(N-J)\times (N-J)$ system $E_{22}$ such that all terms in the
asymptotic expansion of the symbol of $Q_{2\ell}-P_{22}E_{2\ell}$ have vanishing
Taylor coefficients at $\gamma$ for $\ell=1,2$. Then $E=(E_{k\ell})$ has the required properties.

It remains to prove Theorems \ref{thm:sysmainthm1} and \ref{thm:sysmainthm2}.
Since the system $R$ in both results share some properties,
we begin with a general discussion.
First, as in the scalar case we note that in view of the calculus
it suffices to prove the theorems for the adjoint
\begin{equation}\label{eqnew:sysadjoint}
R^\ast (x,D_{x'}) =\left( \begin{array}{cccccc} R_{11} & \ldots & R_{1J} & 0 &\ldots & 0 \\
\vdots & \ddots & \vdots & \vdots & \ddots & \vdots \\
R_{N1} & \ldots & R_{NJ} & 0 &\ldots & 0  \end{array} \right)
\end{equation}
of $R$. Let therefore the symbol of $R^\ast$
have the asymptotic expansion
\begin{equation}\label{eqnew:sysadjoint1}
\sigma_{R^\ast}
\sim \sum_{j=-1}^{\infty}r_{-j},
\end{equation}
where $r_{-j}$ is the homogeneous matrix of degree $-j$ in the asymptotic
sum of the symbol of $R^\ast$. Regarding the Taylor coefficients of
$r_{-j}$ as matrices, we can for any point $(x_0,\xi_0)$ belonging to
$\varGamma'$ then use the ordering $>_t$ given by~\cite[Definition~3.2]{jw}
to find the first nonzero matrix $\mathcal R_0=r_{-j_0 (\alpha_0)}^{ (\beta_0)}(x_0,\xi_0)$
with respect to $>_t$. If $j_0+|\alpha_0|+|\beta_0|=m_0$ for some number $m_0$,
then in particular all matrices $r_{-j (\alpha)}^{ (\beta)}(x_0,\xi_0)$
equal the zero matrix for $j+|\alpha|+|\beta|<m_0$.
Since the ordering will not appear explicitly in
the proof we refrain from describing it further.
We will assume that
we have a nonzero entry in the first row and the first column in the
matrix $\mathcal R_0$,
but this will only affect the construction below in an obvious manner,
so it is of no importance.

\begin{proof}[Proof of Theorem \ref{thm:sysmainthm1}]
We shall prove the theorem by contradiction, arguing
that if it is false, then Lemma \ref{lemrange1}
does not hold. This will be accomplished by
constructing approximate solutions to the equation
$P^\ast v=0$ concentrated near $\varGamma'$ in such a way that
the proof reduces to the scalar case.
Note that the symbol of $R^\ast$
is independent of $\xi_1$, and that $R^\ast$ acting on a vector $v\in C_0^\infty(\mathbb{R}^n,
\C^N)$ only depends on the first $J$ coordinates of $v$. Hence we can let
the approximate solutions
be vectors in $\C^J\times \{0\}\subset \C^N$.
We shall let
each component be an approximate solution to a scalar problem of the same kind,
constructed as in~\cite[Section~4]{jw}.

To simplify notation, we shall in what follows write $t$ instead of $x_1$ and $x$ instead of
$x'=(x_2,\ldots,x_n)$, and we may without loss of generality assume that $\varGamma'$
is given by
\[
\varGamma'=\{(t,0,0,\xi^0): a\le t\le b\},
\]
where $\xi^0=(0,\ldots,0,1)\in\R^{n-1}$.
Let $K$ and $\varOmega$ be the cones given by Theorem \ref{thm:sysmainthm1}.
Given any positive integer $M$ we can by~\cite[Lemma~26.4.14]{ho4} find a curve $t\mapsto (t,y(t),0,\eta(t))$
as close to $\varGamma'$ as desired, and functions $w_0$ and $w_\alpha$ such that
\begin{equation}\label{formofsoln}
w(t,x)=w_0(t)+\langle x-y(t),\eta(t)\rangle+\sum_{2\le|\alpha|\le M}w_\alpha(t)(x-y(t))^\alpha/|\alpha|!
\end{equation}
is a formal solution to the eiconal equation
\begin{equation}\label{eiconal}
\partial w/\partial t - i f(t,x,\partial w/\partial x)=0
\end{equation}
with an error of order $\mathcal{O}(|x-y(t)|^{M+1})$
in a neighborhood $Y$ of
\begin{equation}\label{specialcompactset}
\{(t,0): a\leq t \leq  b\}\subset \mathbb{R}^n,
\end{equation}
such that $\im w>0$ in $Y$ except
on a compact non-empty subset $T$ of the curve $x=y(t)$,
while $w=0$ on $T$. By part (i) of~\cite[Lemma~26.4.14]{ho4}
we can choose $w$ so that
\begin{equation*}
\varGamma_0=\{ (t,x,\partial w(t,x) / \partial t , \partial w(t,x) / \partial x ): (t,x)\in T \}
\end{equation*}
is contained in $\varOmega$, which is done to ensure that if
$A$ is a given system of pseudodifferential operators with
wave front set contained in the complement
of $K$, then $W\! F(A)$ does not meet the cone generated by $\varGamma_0$.
Note also that the functions $w_\alpha$ can be chosen so that
for $|\alpha|=2$ we have that the matrix $\im w_{jk}-\delta_{jk}/2$
is positive definite, where $\delta_{jk}$ is the Kronecker delta.
If $\varGamma'$ is a point
we can thus obtain a sequence $\{\gamma_j\}_{j=1}^\infty$ of curves
\[
\gamma_j(t)=(t, y_j(t),0,\eta_j(t)), \quad a_j'\leq t \leq b_j',
\]
approaching $\varGamma'$ together with solutions $w_j$ to \eqref{eiconal} which implies that
at $t=c_j'$ we have
\[
(c_j', y_j(c_j'),0,\eta_j(c_j'))\rightarrow \varGamma' \quad \text{as } j\rightarrow \infty
\]
in $T^\ast(\mathbb{R}^n)\smallsetminus 0$, where $c_j'$ is the point where
$\re w_{0 j}=\im w_{0 j}=0$. 
Similarly, if $\varGamma'$ is an interval
and $\varrho\ge 0$ is the number given by Theorem \ref{thm:sysmainthm1},
then for any point $\omega$ in the interior of $\varGamma_{\varrho}'$
we can use~\cite[Lemma~4.1]{jw} in place of~\cite[Lemma~26.4.14]{ho4}
to obtain a sequence $\{\gamma_j\}_{j=1}^\infty$ of curves
approaching $\varGamma'$ and a sequence $\{w_{0j}\}_{j=1}^\infty$ of functions
such that for each $j$ there exists a point $\omega_j\in\gamma_j$
with $\omega_j=\gamma_j(t_j)$
which can be chosen so that $\re w_{0 j}(t_j)=\im w_{0 j}(t_j)=0$ and
$\omega_j\rightarrow \omega$ as $j\to \infty$.
If all the terms in the asymptotic sum of the symbol of $R^\ast$ have vanishing Taylor coefficients at
$\omega_j$, or at $(c_j', y_j(c_j'),0,\eta_j(c_j'))$ when $\varGamma'$ is a point, then Theorem
\ref{thm:sysmainthm1}
will follow by continuity. In what follows we will suppress the index $j$
to simplify notation, and we will show that
all the terms in the asymptotic sum of the symbol of $R^\ast$ have vanishing Taylor coefficients at
one of these points, denoted henceforth by $\omega_0$, with
$\omega_0=\gamma(t_0)$ for some curve
\begin{equation}\label{criticalcurve}
t\mapsto\gamma(t)=(t,y(t),0,\eta(t))
\end{equation}
with the properties given above.

So suppose this is false, and let $R^\ast$ be given by \eqref{eqnew:sysadjoint}.
Let $M$ be a large positive integer to be determined later,
and let $w$ be of the form
\eqref{formofsoln}, corresponding to the curve $t\mapsto\gamma(t)$ containing $\omega_0$,
such that $w$ is an approximate solution to \eqref{eiconal} with
an error of order $\mathcal{O}(|x-y(t)|^{M+1})$ in a neighborhood
$Y$ of \eqref{specialcompactset}.
Let $N_0$ be the integer given by Definition \ref{defrangesystems}, and for
$1\le k\le J$ let $v_{k,\tau}\in C_0^\infty(\mathbb{R}^n,\C)$ be an approximate solution
of the form
\begin{equation}\label{eq:scalarapprsol}
v_{k,\tau}(t,x)=e^{i\tau w(t,x)}\sum_{m=0}^M \phi_{k,m}(t,x)\tau^{-m}.
\end{equation}
Here the amplitude functions $\phi_{k,m}\in C_0^\infty(\R^n,\C)$
are to be determined shortly.
Let
\[
V_\tau=\tau^{N_0+n}(v_{1,\tau},\ldots,v_{J,\tau},0)\in C_0^\infty(\mathbb{R}^n,\C^N).
\]
Note that the $v_{k,\tau}$'s are approximate solutions of the same type
as those in~\cite[Section~26.4]{ho4}. Taking
the additional factor $\tau^{N_0+n}$ in $V_\tau$
into account, it therefore follows by~\cite[Lemma~26.4.15]{ho4} that we have
\begin{align}
\label{est:usewithlemma1}
\|V_\tau\|_{(-N_0-n-\kappa)} & \le C\tau^{-\kappa},\quad\tau>1,\\
\label{est:usewithlemma2}
\|A V_\tau\|_{(0)} & \le C\tau^{-\kappa},\quad\tau>1,
\end{align}
for any $\kappa>0$ if $A$ is a pseudodifferential operator with wave front set disjoint
from the cone generated by
\begin{equation}\label{controlofwf}
\{(t,x,w'(t,x)): x\in\bigcup_{k,m}\supp\phi_{k,m}, \im w(t,x)=0\}.
\end{equation}

If $\nu$ is the number given by Lemma \ref{lemrange1}
and $\kappa$ is any positive number,
then our goal is to choose the amplitude functions so that
\begin{equation}\label{est:usewithlemma3}
\|P^\ast V_\tau\|_{(\nu)}\le C\tau^{-\kappa}
\end{equation}
if the number $M$ given by \eqref{eq:scalarapprsol}
is sufficiently large.
Note that this estimate is not affected if the amplitude
functions $\phi_{k,m}$ are multiplied with a
cutoff function in $C_0^\infty(Y,\R)$ which is $1$ in a
neighborhood of the compact set where $\im w=0$. Since
the $\phi_{k,m}$'s will be irrelevant outside $Y$ for large
$\tau$ by construction, we can in this way choose them
to be supported in $Y$ so that $V_\tau\in C_0^\infty(Y,\C^N)$.
Now,
\begin{equation}\label{actionformulaPstar}
P^\ast V_\tau=
\left( \begin{array}{c}
\tau^{N_0+n}P_{11}^\ast {}^t(v_{1,\tau},\ldots ,v_{J,\tau})\\
0
\end{array}\right),
\end{equation}
and by the assumptions of Theorem \ref{thm:sysmainthm1}
we can write $P_{11}^\ast=(D_t-if(t,x,D_x))\mathrm{Id}_J+F_0(t,x,D_{x})$
for some system $F_0\in\varPsi_\mathrm{cl}^0(\R^n)$ with symbol
depending on $t$, $x$ and $\xi$.
Since we have $\im w(t,x)>0$ everywhere except at some points
belonging to the curve $(t,x)=(t,y(t))$ where $w_x'(t,y(t))=\eta(t)$
and $\im w''$ is positive definite by construction,
we can use~\cite[Lemma~26.4.16]{ho4} to obtain a formula
for how $P^\ast$ acts on $V_\tau$.
In view of the discussion following that result,
we find that since $f$ is homogeneous of degree $1$ we have
\[
f(t,x,D_x)(e^{i\tau w}\phi_{k,m})=e^{i\tau w}
\sum_{|\alpha|\le M} f^{(\alpha)}(t,x,\tau w_x')D_x^\alpha\phi_{k,m}
+\mathcal{O}(\tau^{(1-M)/2})
\]
for $1\le k\le J$.
Here $f(t,x,\xi)$ is not defined for complex $\xi$, but since
$w_x'(t,y(t))=\eta(t)$, the expression
$f^{(\alpha)}(t,x,\tau w_x')$ is given meaning if it for each multi-index $\alpha\in\N^{n-1}$
is replaced by a finite Taylor expansion at $\tau \eta(t)$
representing the value at $\tau w_x'(t,x)$.

Now recall that $w$ is an approximate solution to \eqref{eiconal} with an error of order
$\mathcal O(|x-y(t)|^{M+1})$. Since a function $\chi(t,x)e^{i\tau w}$ can be
estimated by $\tau^{-\ell/2}$ if $\chi$ vanishes of order $\ell$ when $x=y(t)$
it follows that
\[
e^{i\tau w}(\tau w_t'- i\tau f(t,x,w_x'))\phi_{k,m}=\mathcal O(\tau^{(1-M)/2}).
\]
Recalling the definition of $v_{k,\tau}$ and using the homogeneity of $f$ we thus obtain
\begin{equation}\label{eq:syspaction0}
(D_t-if(t,x,D))v_{k,\tau}=e^{i\tau w}\sum_{m=0}^M\tau^{-m}\psi_{k,m}
+\mathcal{O}(\tau^{(1-M)/2})
\end{equation}
where
\begin{equation*}
\psi_{k,m}=D_t\phi_{k,m}
-\sum_{1\le|\alpha|\le M} i\tau^{1-|\alpha|}f^{(\alpha)}(t,x,w_x')D_x^\alpha\phi_{k,m}.
\end{equation*}
If $\sigma_{F_0}(t,x,\xi)\sim \sum_{j=0}^\infty f_{-j}(t,x,\xi)$
where $f_{-j} =( f_{k\ell,-j} )$
are $J\times J$ matrices homogeneous of degree $-j$,
then we can use the homogeneity of $f_{-j}$ and apply~\cite[Lemma~26.4.16]{ho4} to obtain
\begin{equation*}
F_0 {}^t(v_{1,\tau},\ldots ,v_{J,\tau})  = e^{i\tau w}
A_\tau(t,x)+\mathcal{O}(\tau^{(1-M)/2})
\end{equation*}
where
\begin{equation}\label{eq:syspaction2}
A_\tau(t,x)=
\left( \begin{array}{c}
\sum_{j,\ell,m,\alpha}
\tau^{-j-|\alpha|-m}
f_{1\ell,-j}^{(\alpha)}(t,x, w_{x}' ) D_{x}^\alpha
\phi_{\ell,m}/\alpha !  \\
\vdots \\
\sum_{j,\ell,m,\alpha}
\tau^{-j-|\alpha|-m}
f_{J\ell,-j}^{(\alpha)}(t,x, w_{x}' ) D_{x}^\alpha
\phi_{\ell,m}/\alpha !   \end{array} \right)
\end{equation}
and the sum is taken over $1\le \ell \le J$ and all
$0\le j\le M'$,
$0\le m\le M$, $|\alpha|<M-1-2j$ for some sufficiently large $M'$
(see equation $(4.21)$ in~\cite{jw}).
In \eqref{eq:syspaction2} we should replace
$f_{k\ell,-j}^{(\alpha)}(t,x, w_{x}' )$ by a Taylor expansion
at $\eta(t)$ as above.
Hence equations \eqref{eq:syspaction0}--\eqref{eq:syspaction2}
imply that
\[
P_{11}^\ast {}^t(v_{1,\tau},\ldots ,v_{J,\tau})=
e^{i\tau w} \left( \begin{array}{c} \sum_{m=0}^M \tau^{-m}
\varPsi_{1,m}  \\
\vdots \\
\sum_{m=0}^M \tau^{-m}
\varPsi_{J,m}  \end{array} \right)+\mathcal{O}(\tau^{(1-M)/2}),
\]
where
\begin{equation*}
\varPsi_{k,m}=D_t\phi_{k,m}
-\sum_{|\alpha|=1}if^{(\alpha)}(t,x,w_x')D_x^\alpha\phi_{k,m}
+\sum_{\ell=1}^J f_{k\ell,0}(t,x, w_{x}') \phi_{\ell,m}
+R_{k,m}
\end{equation*}
with $R_{k,0}=0$ for $1\le k \le J$ and
$R_{k,m}$ determined by $\phi_{\ell,0},\ldots,\phi_{\ell,{m-1}}$,
$1\le \ell\le J$, for $m>0$.
Set
\begin{equation*}
\phi_{k,0}(t,x)=\sum_{|\alpha|<M}\phi_{k,0\alpha}(t)(x-y(t))^\alpha
\end{equation*}
where $y(t)$ is the $x$ coordinate of the curve $t\mapsto\gamma(t)$
in \eqref{criticalcurve} containing the point $\omega_0$.
Then
$\varPsi_{k,0}(t,x)=\mathcal{O}((x-y(t))^M)$ for $1\le k\le J$
if $\phi_{k,0\alpha}$ satisfy a certain linear system of ordinary differential
equations
\begin{equation}\label{ODE}
D_t\phi_{k,0\alpha}+\sum_{\stackrel{1\le \ell\le J}{|\beta|<M}}a_{k\ell,\alpha\beta}\phi_{\ell,0\beta}=0.
\end{equation}
Given any non-negative integer $m_0<M$, these equations may be solved so that, for example,
$D_{x}^\alpha\phi_{k,0}=0$ at $(t_0,y(t_0))$ for all
$|\alpha|\le m_0$ and $2\le k\le J$, while $D_{x}^\alpha\phi_{1,0}(t_0,y(t_0))=0$ for all
$|\alpha|\le m_0$ except for one index $\alpha_0$ with $|\alpha_0|=m_0$.
We may in the same way successively choose $\phi_{k,m}$ for $1\le k\le J$ so that
\[
\varPsi_{k,m}(t,x)=\mathcal{O}((x-y(t))^{M-2m})\quad\text{when }m<M/2.
\]
Using again the fact that a function of the form $\chi(t,x) e^{i\tau w}$ can be estimated by $\tau^{-\ell/2}$
if $\chi$ vanishes of order $\ell$ when $x=y(t)$, it follows that
if $M$ is chosen so that $(1-M)/2\le -N_0-n-\nu-\kappa$, then we obtain
$P^\ast V_\tau=\mathcal{O}(\tau^{-\nu-\kappa})$
in view of \eqref{actionformulaPstar}.
By the discussion in~\cite[p.~110]{ho4}) we conclude
that for any integer $\kappa$ we can
find a constant $C$ such that
\eqref{est:usewithlemma3} holds
if only $M=M(\kappa)$ is chosen sufficiently large.

Recall that $R^\ast$ is given by \eqref{eqnew:sysadjoint}, and let
the symbol of $R^\ast$ have the asymptotic expansion given by \eqref{eqnew:sysadjoint1}.
Since we will prove Theorem \ref{thm:sysmainthm1} by contradiction,
suppose that $\mathcal R_0=r_{-j_0 (\alpha_0)}^{ (\beta_0)}(\omega_0)$
is the first nonzero matrix with respect to the ordering $>_t$
given by~\cite[Definition~3.2]{jw}, where
\begin{equation}\label{eq:defofm0}
j_0+|\alpha_0|+|\beta_0|=m_0.
\end{equation}
Here $\omega_0=(t_0,y(t_0),0,\eta(t_0))$.
As mentioned above we will assume that
we have a nonzero entry in the first row and the first column in the
matrix $\mathcal R_0$.
Now let $H\in C_0^\infty(\mathbb{R}^n,\C)$, and
define $h_\tau:\mathbb{R}^n\rightarrow\C$ by
\[
h_\tau(t,x)= H(\tau (t-t_0),\tau (x-y(t))).
\]
With $\mathcal H_\tau:\mathbb{R}^n\rightarrow\C^N$ given by
$\mathcal H_\tau=\tau^{-N_0}(h_\tau,0)$ it follows by~\cite[Proposition~4.3]{jw}
that for $\tau\ge 1$ we have $\mathcal H_\tau\in
H_{(N_0)}(\mathbb{R}^n,
\C^N)$ and $\|\mathcal H_\tau\|_{(N_0)}\le C$ where the
constant depends on $H$ but not on $\tau$.
In fact, the proof shows that $\|\mathcal H_\tau\|_{(N_0)}\le C\tau^{-n/2}$
for $\tau\ge 1$ but this is not needed.
(If we have a nonzero entry on the $i$:th row in the matrix
$\mathcal R_0$, then choose $\mathcal H_\tau$ as above with $h_\tau$ on the $i$:th coordinate.)
Then
\begin{equation}\label{eq:integralsystem}
(R^\ast V_\tau, \overline{\mathcal H_\tau} )_{L^2(\R^n,\C^N)}
=\sum_{k=1}^J \tau^n
( R_{1k}v_{k,\tau},\overline{h_\tau}),
\end{equation}
where $(\phantom{i},\phantom{i})$ denotes the usual scalar product on ${L^2(\R^n,\C)}$,
and by Lemma \ref{lemrange1} applied to the system $R$ together with equations
\eqref{est:usewithlemma1}, \eqref{est:usewithlemma2} and \eqref{est:usewithlemma3},
the left-hand side can be estimated by $C_\kappa\tau^{-\kappa}$ for any $\kappa$.
As in the proof of~\cite[Theorem~2.21]{jw} we want to determine the limit of
\[
\tau^{m_0}(R^\ast V_\tau, \overline{\mathcal H_\tau} )_{L^2(\R^n,\C^N)}
\]
as $\tau\to\infty$ with $m_0$ given by \eqref{eq:defofm0},
and show that if the terms of the symbol of $R^\ast$ do
not all vanish of infinite order at $\omega_0$
then $H$ can be chosen so that this limit
is nonzero, which is the contradiction that proves the theorem.
For each integral in the right-hand side of \eqref{eq:integralsystem}
we can use~\cite[Lemma~26.4.16]{ho4} and homogeneity
to obtain an auxiliary formula
for \eqref{eq:integralsystem} as an asymptotic series
in $\tau$, where the coefficients consist among other things
of derivatives in $x$ of the amplitude functions $\phi_{k,m}$. After
the change of variables $(\tau (t-t_0),\tau (x-y(t)))\mapsto(t,x)$ we Taylor expand each
term in the asymptotic sum to sufficiently high order, and then sort the result in
declining homogeneity degree in $\tau$
(see equations (4.21)-(4.23) together with (4.33) in~\cite{jw}, and note
that there, $t_0$ is assumed to be $0$).
If $\pi:T^\ast(\R^n)\to\R^n$ is the projection onto
the base manifold, and we for $2\le k\le J$ choose $\phi_{k,0}$ to have vanishing Taylor coefficients
with respect to the $x$ variable
at $\pi(\omega_0)=(t_0,y(t_0))$ of sufficiently high order,
then in view of equation
(4.34) in~\cite{jw}
we see that the only contribution in \eqref{eq:integralsystem}
will come from $(Q_{11}v_{1,\tau}, \overline{h_\tau})$.
(If on the $i$:th row we have a nonzero entry in the $j$:th column
in $\mathcal R_0$, choose $\phi_{k,0}$ as above for all $k\ne j$.)
Since this reduces the situation to the scalar case,
the theorem follows by repeating the proof of~\cite[Theorem~2.21]{jw}.
\end{proof}

We now prove Theorem \ref{thm:sysmainthm2} using the same strategy
as the one used to prove Theorem \ref{thm:sysmainthm1}.

\begin{proof}[Proof of Theorem \ref{thm:sysmainthm2}]
We first construct approximate solutions to the equation
$P^\ast v=0$ concentrated near $\varGamma'=\{(0,\varepsilon_n)\}$.
As in the proof of Theorem \ref{thm:sysmainthm1} we can let
the approximate solutions
be vectors in $\C^J\times \{0\}\subset \C^N$,
and we will again let
each component be an approximate solution to a scalar problem of the same kind,
constructed this time as in~\cite[Section~3]{jw}.
Thus, for $1\le k\le J$ let $v_{k,\tau}\in C_0^\infty(\mathbb{R}^n,\C)$
be an approximate solution of the form
\[
v_{k,\tau}(x)=\phi_{k}(x)e^{i\tau w(x)}
\]
where
\begin{equation}\label{specialw}
w(x)=x_n+i(x_1^2+x_2^2+\ldots+x_{n-1}^2+(x_n+ix_1^2/2)^2)/2
\end{equation}
is a solution to $P^\ast w=0$
and $\phi_k\in C_0^\infty(\mathbb{R}^n,\C)$.
By the Cauchy-Kovalevsky theorem we can solve $D_1\phi_k-ix_1D_n\phi_k=0$ in a neighborhood of $0$
for any analytic initial data $\phi_k(0,x')=f_k(x')\in C^\omega(\mathbb{R}^{n-1},\C)$;
in particular we are free to specify the Taylor coefficients of $f_k(x')$ at $x'=0$. For
$1\le k\le J$ we
take $\phi_k$ to be such a solution. If need be
we can reduce the support of each $\phi_k$ by multiplying by a smooth cutoff function $\chi$
where $\chi$ is equal to $1$ in some
smaller neighborhood of $0$ so that $\chi\phi_k$ solves the equation there. We assume this to be done
and note that if the support of each $\phi_k$ is small enough then
\begin{equation*}
\im w(x)\geq|x|^2/4, \quad x\in \bigcup_k\supp\phi_k.
\end{equation*}
Since
\[
d \re w(x)=-x_1x_n d x_1 + (1-x_1^2/2) d x_n
\]
we may similarly assume that
$d \re w(x)\ne 0$ in $\supp\phi_k$, $1\le k\le J$.
If $V_\tau=\tau^{N_0+n}(v_{1,\tau},\ldots,v_{J,\tau},0)\in C_0^\infty(\mathbb{R}^n,\C^N)$,
then by~\cite[Lemma~26.4.15]{ho4} it follows that for any $\kappa>0$ there is a constant $C$
such that
\begin{align}
\label{eq:bound1}
\|V_\tau\|_{(-N_0-n-\kappa)} & \le C\tau^{-\kappa},\quad\tau>1,\\
\|A V_\tau\|_{(0)} & \le C\tau^{-\kappa},\quad\tau>1,
\label{eq:bound2}
\end{align}
if $A$ is a pseudodifferential operator with wave front set disjoint
from the cone generated by
\[
\{(x,w'(x)): x\in\bigcup_{k}\supp\phi_{k}, \im w(x)=0\}.
\]
Since
\[
P^\ast V_\tau=
{}^t\Big((D_1-ix_1D_n)(e^{i\tau w}\phi_1),\ldots,
(D_1-ix_1D_n)(e^{i\tau w}\phi_J),0,\ldots,0\Big)
\]
by construction, it follows that
\begin{equation}\label{eq:bound3}
\tau^{m}\|P^\ast V_\tau\|_{(\nu)}\to 0\quad\text{as }\tau\to\infty
\end{equation}
for any positive integers $m$ and $\nu$ by~\cite[Lemma~3.1]{jw}.

Now note that
if we write $t$ instead of $x_1$ and $x$ instead of $x'$, then the solution
$w$ to $(D_1-ix_1D_n)w=0$ given by \eqref{specialw} takes the form
\begin{equation}\label{eq:identifyspecialcase}
w(t,x)=i(t^2-t^4/4)/2+\langle x, (1-t^2/2)\xi^0\rangle+i|x|^2/2,
\end{equation}
where as usual $\xi^0=(0,\ldots,1)\in\mathbb{R}^{n-1}$.
Comparing this to the solution of the eiconal equation given by \eqref{formofsoln},
we see that \eqref{eq:identifyspecialcase} is the special case
$w_0(t)=i(t^2-t^4/4)/2$, $y(t)\equiv 0$, $\eta(t)=(1-t^2/2)\xi^0$ and
$w_\alpha(t)\equiv 0$
for $|\alpha|\ge 3$,
$w_{jk}(t)=i\delta_{jk}$
where $\delta_{jk}$ is the Kronecker $\delta$. 
Thus, $t\mapsto (t,y(t),0,\eta(t))$ is a curve through the point
$\varGamma'$. Having established the estimates
\eqref{eq:bound1}--\eqref{eq:bound3}, Theorem \ref{thm:sysmainthm2}
therefore follows if we repeat the end of the proof of
Theorem \ref{thm:sysmainthm1}.
We omit the details.
\end{proof}

In view of the construction of approximate solutions to $P^\ast v=0$
in the proof of Theorem \ref{thm:sysmainthm1}, we can now give a
short proof of Theorem \ref{thm:solv}.

\begin{proof}[Proof of Theorem \ref{thm:solv}]
Let $K$ be the cone generated by $\gamma(I)$
and recall that we only have to verify the theorem when $\gamma(I)$
is a minimal bicharacteristic interval, that is, when case (b) holds.
In view of Proposition \ref{prop.26.4.4} with $Q=\I_N$
we may assume that $P$ has the block form given Lemma \ref{lem:prep},
with the principal symbol of the $J\times J$ system $P_{11}$ satisfying
$\sigma(P_{11})(w)=\lambda(w)\I_J$ where $\lambda(w)$ is the section
of eigenvalues of $P$ given by Theorem \ref{thm:solv}.
In fact,
since the systems $A$ and $B$ in Lemma \ref{lem:prep} are
homogeneous and non-characteristic in a neighborhood of $\gamma(I)$,
we can find a microlocal parametrix $E$ of $AQB=AB$
such that
\[
W\! F(EAB-\I_N)\cap K= W\! F(ABE-\I_N)\cap K=\emptyset.
\]
Applying Proposition \ref{prop.26.4.4}
shows that $P$ is solvable at $K$ if and only if
the range of $AB$ is microlocally contained
in the range of $APB$ at $K$,
and using the existence of $E$
it is easy to see that the latter holds if and only if
$APB$ is solvable at $K$.
(Alternatively, the proof of~\cite[Proposition~26.4.4]{ho4}
immediately generalizes to a proof for a corresponding
result for square systems, so this could be
used in place of Proposition \ref{prop.26.4.4}.)
Keeping this observation in mind,
we can in view of
Definition \ref{dfn:minimal1} then use
Lemma \ref{sys:symbolvanishafterconjugation} in the appendix,
again with $Q=\I_N$,
to further reduce the proof to the case when $P\in\varPsi_{\mathrm{cl}}^1(\R^n)$,
$\lambda(x,\xi)=\xi_1+if(x,\xi')$ and
\[
\gamma(I)=\{(x_1,0,\varepsilon_n):x_1\in I\},
\]
where $f$ is real valued, homogeneous of degree $1$ and independent of $\xi_1$.
Since the normal form of $\lambda(x,\xi)$ is only valid in a neighborhood
of $\{(x_1,0,\varepsilon_n):x_1\in I\}$ we actually have to use a pseudodifferential
cutoff for this to hold, but this can be accomplished by adapting the
arguments in~\cite[pp.~107-108]{ho4}.

Note that $\gamma$ is a minimal bicharacteristic interval of $\lambda(w)$, so
in every neighborhood of $\gamma(I)$ there is a bicharacteristics of $\re\lambda=\xi_1$
along which $f$ changes sign from $-$ to $+$, and $f$ vanishes of infinite order on
$\gamma(I)$. Since we are assuming that $|I|>0$ there is a neighborhood of
$\gamma(I)$ where \eqref{eq:sysmainthm1} holds. This is all that is required
for us to repeat the construction of the approximate solutions to $P_{11}^\ast v=0$
from the proof of Theorem \ref{thm:sysmainthm1}, so let
$V_\tau=(v_{1,\tau},\ldots,v_{J,\tau},0)\in C_0^\infty(\R^n,\C^N)$
be the corresponding approximate solution to $P^\ast V=0$. Assume to reach a contradiction
that $P$ is solvable at the cone $K$ generated by $\gamma(I)$ and let $N_0$
be the integer given by Definition \ref{defrangesystems} with $Q=\I_N$.
If $A$ is the system given by Lemma \ref{lemrange1}
such that $W\! F(A)\cap K=\emptyset$,
concentrate $V_\tau$ so close to $\gamma$ so that $W\! F(A)$
does not meet the cone generated by \eqref{controlofwf}.
Note that $V_\tau$ differs from the approximate
solutions in the proof of Theorem \ref{thm:sysmainthm1} by a factor of $\tau^{-N_0-n}$.
In any case, equations
\eqref{est:usewithlemma1}, \eqref{est:usewithlemma2} and \eqref{est:usewithlemma3}
imply that $V_\tau$ can be constructed so that the right-hand side of \eqref{rangeeq1} is bounded
by $C\tau^{-\kappa}$ for any $\kappa$ if $\tau>1$.
Finally, by the discussion following \eqref{ODE} we can choose at least one
of the amplitude functions $\phi_{k,0}$ in the definition of the
$v_{k,\tau}$'s to be non-vanishing at an appropriately chosen point,
which by~\cite[Lemma~26.4.15]{ho4} implies that $\|V_\tau\|_{(-N_0)}\ge c\tau^{-n/2-N_0}$
for some $c>0$. Applying Lemma \ref{lemrange1} with $Q=\I_N$ we obtain a contradiction,
which completes the proof.
\end{proof}

\appendix

\section{}
\label{appendix1}
Here we prove a few results used in the main text, related to how the property
that all terms in the asymptotic expansion of the total symbol have vanishing
Taylor coefficients is affected by various operations.
Some of these results are straightforward generalizations of the
corresponding results for the scalar case, see~\cite[Appendix~A]{jw}.

\begin{lem}\label{sys:symbolvanishafterconjugation}
Suppose $X$ and $Y$ are two $C^\infty$ manifolds of the same dimension $n$.
Let $K\subset T^\ast(X)\smallsetminus 0$ and
$K'\subset T^\ast(Y)\smallsetminus 0$ be compactly based cones and let $\chi$ be a
homogeneous symplectomorphism from a conic neighborhood of $K'$ to one of $K$ such that
$\chi(K')=K$, and let $\varGamma$ be the graph of $\chi$.
Let $P\in\varPsi_{cl}^m(Y)$ be an $N\times N$ system of
properly supported classical pseudodifferential
operators in $Y$ of the form
\begin{equation}\label{eq:1stnormalforsystems2}
P=\left( \begin{array}{cc} P_{11} & 0\\
0 & P_{22} \end{array} \right)
\end{equation}
where the principal symbol of the $J\times J$ system $P_{11}$ is given by $\sigma(P_{11})=\lambda\mathrm{Id}_J$
for some scalar function $\lambda\in C^\infty(T^\ast(Y)\smallsetminus 0)$,
homogeneous of degree $m$,
and $P_{22}$ is an $(N-J)\times (N-J)$ system, elliptic in a conic neighborhood of $K'$.
Suppose that there exists a function $0\ne q\in C^\infty(T^\ast(Y)\smallsetminus 0)$
such that
\[
(\chi^{-1})^\ast (q\lambda)=\xi_1+if(x,\xi').
\]
Then one can find $N\times N$ systems
$A\in I_\mathrm{cl}^{1-m}(X\times Y, \varGamma')$ and $B\in I_\mathrm{cl}^{0}(Y\times X,(\varGamma^{-1})')$
of properly supported Fourier integral operators such that
\begin{itemize}
\item[(i)] $A$ and $B$ are
non-characteristic at the restriction of the graphs of $\chi$ and $\chi^{-1}$ to $K'$ and to $K$ respectively,
while $W\! F'(A)$ and $W\! F'(B)$ are contained in small conic neighborhoods,
\item[(ii)] $APB\in\varPsi_\mathrm{cl}^1(X)$ has the form \eqref{eq:1stnormalforsystems2}
with $P_{jj}$ replaced by $\tilde{P}_{jj}$ for $j=1,2$,
where $\sigma(\tilde{P}_{11})=(\xi_1+if(x,\xi'))\mathrm{Id}_J$
and $\tilde{P}_{22}$ is elliptic in a conic neighborhood of $K$.
\end{itemize}
Moreover, if $R$ is an $N\times N$ system of
properly supported classical pseudodifferential
operators in $Y$, then each term in the asymptotic expansion of the symbol of $R$
has vanishing Taylor coefficients at a point $(y,\eta)\in K'$ if and only if
each term in the asymptotic expansion of the symbol of the pseudodifferential operator
$ARB$ in $X$ has vanishing Taylor coefficients at $\chi(y,\eta)\in K$.
\end{lem}

\begin{proof}
Let $P_{11}=(Q_{jk})$
and choose any properly supported scalar
Fourier integral operators
$A\in I_\mathrm{cl}^{1-m}(X\times Y, \varGamma')$ and $B\in I_\mathrm{cl}^{0}(Y\times X,(\varGamma^{-1})')$
such that the principal symbol of $BA$ is equal to $q$ in a conic neighborhood $\varOmega$ of
$K'$. Since $q\ne 0$ we find that $A$ and $B$ are non-characteristic
at the restriction of the graphs of $\chi$ and $\chi^{-1}$ to $K'$ and to $K$ respectively.
We may choose $A$ and $B$ such that
$W\! F'(A)$ and $W\! F'(B)$ are contained in small conic neighborhoods.
Since $\sigma(P_{11})=\lambda\mathrm{Id}_J$ it follows that the principal symbol of
$AQ_{kk}B$ is equal to $\xi_1+if(x,\xi')$
in a neighborhood of $K$ for $1\le k\le J$.

Now choose $A'\in I_\mathrm{cl}^{0}(X\times Y, \varGamma')$
and $B'\in I_\mathrm{cl}^{m-1}(Y\times X,(\varGamma^{-1})')$
properly supported and such that
\begin{align*}
K\cap W\! F(AB'-\mathrm{Id})& = \emptyset,  & K'  \cap W\! F(B'A-\mathrm{Id}) = \emptyset, \\
K\cap W\! F(A'B-\mathrm{Id})& = \emptyset,  & K'  \cap W\! F(BA'-\mathrm{Id}) = \emptyset.
\end{align*}
Naturally, these conditions continue to hold with $\I$ replaced by $\I_N$
if $A$ is replaced by $A\I_N$, and $A'$, $B$ and $B'$ are similarly replaced
by $N\times N$ systems.
The systems $\tilde A=A\I_N$ and $\tilde B=B\I_N$ thus constructed satisfy (i),
and it is also clear that (ii) holds.
Moreover, if $R=(R_{jk})$ is an $N\times N$ system of
properly supported classical pseudodifferential
operators in $Y$ such that each term in the asymptotic expansion of the symbol of $R$
has vanishing Taylor coefficients at a point $(y,\eta)\in K'$,
then each term in the asymptotic expansion of the symbol of $\tilde AR\tilde B=(AR_{jk}B)$ has
vanishing Taylor coefficients at the point $\chi(y,\eta)\in K$
by~\cite[Lemma~A.1]{jw}
applied to $AR_{jk}B$ for $j,k=1,\ldots, N$.
Conversely, if each term in the asymptotic expansion of the symbol of $\tilde AR\tilde B$
has vanishing Taylor coefficients at a point $\chi(y,\eta)\in K$,
then the same argument shows that
each term in the asymptotic expansion of the symbol of $\tilde B'\tilde AR\tilde B\tilde A'$
has vanishing Taylor coefficients at the point $(y,\eta)\in K'$,
where $\tilde A'=A'\I_N$ and $\tilde B'=B'\I_N$.
Since $\tilde B'\tilde AR\tilde B\tilde A'\equiv R$ mod $\varPsi^{-\infty}$ near $K'$,
this completes the proof.
\end{proof}

Let $\{e_k : \ k=1,\ldots, n\}$ be a basis for $\mathbb{R}^n$, let $(U,x)$ be local coordinates on
a smooth manifold $X$ of dimension $n$, and let
\[
\Big\{\frac{\partial}{\partial x_k}:\ k=1,\ldots,n\Big\}
\]
be the induced local frame for the tangent bundle $TX$. 
For a matrix valued function $f\in C^\infty (U,\mathcal{L}_N)$
we can  
use standard multi-index notation to express the partial derivatives of $f$
since the local frame fields commute.
If $\alpha\in\mathbb{N}^n$ is a multi-index we shall by
$\partial_x^\alpha f(\gamma)$ denote the matrix
$(\partial_x^\alpha f_{i j}(\gamma))$ if
$f(\gamma)=(f_{i j}(\gamma))$.

\begin{lem}\label{appthm19}
Let $X$ be a smooth manifold of dimension $n$, and
for $j\ge 1$ let $p,q_j,g_j\in C^\infty(X)$
be $N\times N$ systems.
Let $\{\gamma_j\}_{j=1}^\infty$ be a sequence in $X$
such that $\gamma_j\to\gamma$ as $j\to\infty$, and assume that
$p(\gamma)=p(\gamma_j)=0$ for all $j$.
Assume also that $p$ is of principal type at $\gamma$, that is, there exists a
tangent vector $\partial_\nu\in T_\gamma X$ such that
\[
\partial_\nu p(\gamma): \kernel p(\gamma)\longrightarrow \cokernel p(\gamma)=\C^N/\ran p(\gamma)
\]
is bijective. Let $(U,x)$ be local coordinates on $X$ near $\gamma$,
and suppose that there exists an $N\times N$ system $q\in C^\infty(X)$
such that
\[
\partial_x^\alpha q(\gamma)
=\lim_{j\to\infty} \partial_x^\alpha q_j(\gamma_j)
\]
for all $\alpha\in\mathbb{N}^n$.
If $q_j-pg_j$ vanishes of infinite order at $\gamma_j$ for all $j$, then there
exists an $N\times N$ system $g\in C^\infty(X)$
such that $q-pg$ vanishes of infinite order at $\gamma$.
Furthermore,
\begin{equation}\label{eq:app20}
\partial_x^\alpha g(\gamma)
=\lim_{j\to\infty} \partial_x^\alpha g_j(\gamma_j)
\end{equation}
for all $\alpha\in\mathbb{N}^n$.
\end{lem}

\noindent Note that in view of Borel's theorem,
the assumption concerning the existence of $q$
is equivalent to assuming that all the limits
$\lim_{j\to\infty} \partial_x^\alpha q_j(\gamma_j)$ exist.

\begin{proof}
First note that although the result is stated for a manifold,
it is purely local so we may assume that $X\subset\mathbb{R}^n$ in the proof.
Next we observe
that $p(\gamma)=0$ implies that $\kernel p(\gamma)=\C^N
=\cokernel p(\gamma)$.
Thus $\partial_\nu p(\gamma)$ is invertible, so $|\partial_\nu p(\gamma)|\ne 0$
which means we can find a neighborhood
$\mathcal{U}$
of $\gamma$ where $|\partial_\nu p(\gamma)|\ne 0$. Hence the matrix valued function $\partial_\nu p(w)$ is
invertible in $\mathcal{U}$, and we let $(\partial_{\nu} p(w))^{-1}$ denote its inverse.
By Cramer's rule it follows that $(\partial_{\nu} p)^{-1}$ is $C^\infty$ in $\mathcal{U}$.
We may without loss of generality assume
that $\gamma_j\in\mathcal{U}$ for $j\ge 1$.

Moreover, we have that $\partial_{\lambda\nu} p(\gamma)=\lambda\partial_\nu p(\gamma)$
is invertible for any $0\ne\lambda\in\mathbb{R}$ so we may assume that $\nu$ as a vector in $\mathbb{R}^n$
has length $1$.
(We will identify a tangent vector $\nu\in\mathbb{R}^n$ at $\gamma$ with
$\partial_\nu\in T_\gamma \mathbb{R}^n$
through the usual vector space isomorphism.)
By an orthonormal change of coordinates we may even assume that
$\partial_\nu p(w)=\partial_{e_1}p(w)$.
In accordance with the notation used in the statement of the lemma, we shall
write $\partial_{x_k}p(w)$ for the partial derivatives $\partial_{e_k}p(w)$
and denote by $(\partial_{x_1} p(w))^{-1}$ the inverse of $\partial_{\nu} p(w)=\partial_{x_1} p(w)$
in $\mathcal{U}$.
Now
\begin{equation}\label{eq:app20.5}
0=\partial_{x_1} (q_j-pg_j)(\gamma_j)=\partial_{x_1} q_j(\gamma_j) -  \partial_{x_1} p(\gamma_j) g_j(\gamma_j)
\end{equation}
for all $j$ since $p(\gamma_j)=0$. Since $\lim_j \partial_{x_1} q_j(\gamma_j)=\partial_{x_1} q(\gamma)$
by assumption, equation \eqref{eq:app20.5} yields
\begin{equation*}
\lim_{j\to\infty} g_j(\gamma_j)=(\partial_{x_1} p(\gamma))^{-1}\partial_{x_1} q(\gamma)=a\in\mathcal{L}_N,
\end{equation*}
and we claim that we can in the same way determine
\[
\lim_{j\to\infty} \partial_x^\alpha g_j(\gamma_j)=a_{(\alpha)}\in\mathcal{L}_N
\]
for any $\alpha\in\mathbb{N}^n$.
In fact, arguing by contradiction, we introduce a total
well-ordering of the derivatives $\partial_x^\alpha$ by means of
a monomial ordering of the corresponding monomials $x^\alpha$.
We choose the graded reverse lexiographic order $>_{grevlex}$ together
with the (non-conventional) ordering $x_n>\ldots>x_1$ of
the variables. That is to say, to determine if $\partial_x^\alpha>_{grevlex}\partial_x^\beta$
for multi-indices $\alpha,\beta\in\N^n$,
we first compare the total lengths $|\alpha|$ and $|\beta|$, and in
case of equality compare the left-most entries $\alpha_1$ and $\beta_1$, but
reversing the outcome so that the multi-index with the smaller entry yields
a larger derivative in the ordering. In case of a tie this is followed by a similar
comparison of the second entries from the left and so forth, ending
with a comparison of the right-most entries. This will then lead to
the ordering $x_n>\ldots>x_1$ of the variables, in the sense that the
tangent vectors are ordered $\partial_{x_n}>_{grevlex}\ldots>_{grevlex}\partial_{x_1}$.
Suppose now that $\partial_x^\alpha$ is the first derivative such that
the limit of $\partial_x^\alpha g_j(\gamma_j)$ does not exist as $j\to\infty$.
Let $\varepsilon_k$ be the $k$:th basis vector in $\R^n$, and consider
the limit of $\partial_x^{\alpha+\varepsilon_1}(q_j-pg_j)(\gamma_j)$
as $j\to\infty$. By Leibniz's formula we have
\[
\partial_x^{\alpha+\varepsilon_1}(q_j-pg_j)
=\partial_x^{\alpha+\varepsilon_1}q_j-
p\partial_x^{\alpha+\varepsilon_1}g_j-
\partial_{x_1}p\partial_x^\alpha g_j-
\sum_{\{\beta:\beta<\alpha\}}\binom{\alpha}{\beta}\partial_{x_1}(\partial_x^{\alpha-\beta}p\partial_x^{\beta}g_j).
\]
Note that if $\alpha=(\alpha_1,\ldots,\alpha_n)$ and $\alpha_1\ge 1$,
then the sum over $\beta$ in the right-hand side contains an additional term of the form
$\alpha_1\partial_{x_1}p\partial_x^\alpha g_j$, produced by the value $\beta=\alpha-\varepsilon_1$.
Evaluating at $\gamma_j$ we find that the left-hand side converges to $0$ as $j\to\infty$
by assumption, and since $p(\gamma_j)=0$ it follows from our choice of ordering
that with the exception of the term $(\alpha_1+1)\partial_{x_1}p\partial_x^\alpha g_j$, all
other terms have well-defined limits as $j\to\infty$. Arguing as in the discussion
following \eqref{eq:app20.5}, we can therefore determine the limit of $\partial_x^\alpha g_j(\gamma_j)$
as $j\to\infty$ by multiplying with $(\alpha_1+1)^{-1}(\partial_{x_1} p(\gamma_j))^{-1}$ from the left.
This contradiction proves the claim.

By using Borel's theorem for each entry it is clear that there exists a
matrix valued function $g\in C^\infty(X,\mathcal{L}_N)$ such that
\begin{equation*}
\partial_x^\alpha g(\gamma)=a_{(\alpha)}
=\lim_{j\to\infty} \partial_x^\alpha g_j(\gamma_j)
\end{equation*}
for all $\alpha\in\mathbb{N}^n$. Since $q-pg$ vanishes of infinite order at $\gamma$ by
construction, this completes the proof.
\end{proof}

Keeping the notation from Lemma \ref{appthm19},
there is naturally an analogue result if $p$ is an elliptic system.
In fact, very little has to be changed for the proof to work in
this setting: we essentially just replace $w\mapsto (\partial_{\nu} p(w))^{-1}$
with the inverse $w\mapsto p(w)^{-1}$ of $p$. The ordering used in the proof
can be the same; the only feature needed in this case is that it is a graded
ordering. The result is stated below
for easy reference. We omit the proof.

\begin{lem}\label{appthm19elliptic}
Let $X$ be a smooth manifold of dimension $n$, and
for $j\ge 1$ let $p,q_j,g_j\in C^\infty(X)$
be $N\times N$ systems.
Let $\{\gamma_j\}_{j=1}^\infty$ be a sequence in $X$
such that $\gamma_j\to\gamma$ as $j\to\infty$, and assume that
$|p(\gamma)|$ and $|p(\gamma_j)|$ are non-vanishing for all $j$,
where $|p|$ is the determinant of $p$.
Let $(U,x)$ be local coordinates on $X$ near $\gamma$,
and suppose that there exists an $N\times N$ system $q\in C^\infty(X)$
such that
\[
\partial_x^\alpha q(\gamma)
=\lim_{j\to\infty} \partial_x^\alpha q_j(\gamma_j)
\]
for all $\alpha\in\mathbb{N}^n$.
If $q_j-pg_j$ vanishes of infinite order at $\gamma_j$ for all $j$, then there
exists an $N\times N$ system $g\in C^\infty(X)$
such that $q-pg$ vanishes of infinite order at $\gamma$.
Furthermore,
\begin{equation*}
\partial_x^\alpha g(\gamma)
=\lim_{j\to\infty} \partial_x^\alpha g_j(\gamma_j)
\end{equation*}
for all $\alpha\in\mathbb{N}^n$.
\end{lem}

The method used in the proof of Lemma \ref{appthm19} can also
be applied to prove the following result for certain functions of
non-principal type. We only need
the result for scalar functions but combined with the
first part of the proof of Lemma \ref{appthm19},
the proof would work equally well for systems.

\begin{lem}\label{appthm19special}
Let $X$ be a smooth manifold of dimension $n$, and
let $\lambda$ and $e$ be scalar functions in $C^\infty(X)$.
Let $\gamma\in X$ and assume that
$\lambda(\gamma)=0$ and $d\lambda(\gamma)\ne 0$. If
$\lambda^me$ vanishes of infinite order
at $\gamma$ for some $m\ge 1$, then $e$
vanishes of infinite order at $\gamma$.
\end{lem}

\begin{proof}
As in the proof of Lemma \ref{appthm19} we conclude that since
the statement is local we may assume that $X\subset\mathbb{R}^n$
and $\partial\lambda(\gamma)/\partial x_1\ne 0$, where we use coordinates
$x_1,\ldots,x_n$ in $X$. Let $>_{grevlex}$ be the total well-ordering
of the derivatives $\partial_x^\alpha$ introduced in the proof of
Lemma \ref{appthm19}, and assume that $\partial_x^\alpha e(\gamma)$ is
the first derivate of $e$ that does not vanish at $\gamma$.
Let $\varepsilon_j$ be the $j$:th basis vector in $\R^n$, and consider
the derivative $\partial_{x_1}^m\partial_x^{\alpha}(\lambda^me)(\gamma)$.
By Leibniz's formula we have
\begin{align*}
\partial_{x_1}^m\partial_x^{\alpha}(\lambda^me)
&=\sum_{k=0}^m\binom{m}{k}\partial_{x_1}^k(\lambda^m)\partial_x^{\alpha+(m-k)\varepsilon_1} e\\
&\phantom{=}+\sum_{k=0}^m\binom{m}{k}\sum_{\{\beta:\beta<\alpha\}}\binom{\alpha}{\beta}
\partial_x^{\alpha-\beta+k\varepsilon_1}(\lambda^m) \partial_x^{\beta+(m-k)\varepsilon_1}e.
\end{align*}
In the first sum, all terms with $k<m$ vanish at $\gamma$ since $\lambda(\gamma)=0$,
so the only contribution we get is
$m!(\partial\lambda(\gamma)/\partial x_1)^m\partial_x^\alpha e(\gamma)$.
This also implies that all terms in the double sum with $|\beta|+m-k>|\alpha|$ vanish
at $\gamma$. Conversely, if $|\beta|+m-k<|\alpha|$ then
$\partial_x^{\beta+(m-k)\varepsilon_1}e(\gamma)=0$ since the ordering is graded.
When we have equality $|\beta|+m-k=|\alpha|$ in the double sum then $m-k\ge 1$
since $\beta<\alpha$, so we can write
\begin{equation}\label{betarepresentation}
\beta+(m-k)\varepsilon_1=\alpha-\sum_{\ell=1}^{m-k}\varepsilon_{j_\ell}+(m-k)\varepsilon_1
\end{equation}
for some $\varepsilon_{j_\ell}$ with $1\le {j_\ell}\le n$.
Unless the left-most entry $\alpha_1$ of $\alpha$ is $\ge m-k$
so that we can choose $\varepsilon_{j_\ell}=\varepsilon_1$
for all $1\le {j_\ell}\le n$ in \eqref{betarepresentation},
we thus have $\partial_x^{\alpha}>_{grevlex}\partial_x^{\beta+(m-k)\varepsilon_1}$
which by our assumptions implies that $\partial_x^{\beta+(m-k)\varepsilon_1}e(\gamma)=0$.
On the other hand, if $\beta+(m-k)\varepsilon_1=\alpha$ then
$\partial_x^{\alpha-\beta+k\varepsilon_1}(\lambda^m)=\partial_{x_1}^m(\lambda^m)$
so this produces another term of the form
$m!(\partial\lambda(\gamma)/\partial x_1)^m\partial_x^\alpha e(\gamma)$.
Hence
\[
0=\partial_{x_1}^m\partial_x^{\alpha}(\lambda^me)|_\gamma
= C(\partial\lambda(\gamma)/\partial x_1)^m\partial_x^\alpha e(\gamma)
\]
where $C$ is a positive constant depending only on $m$ and $\alpha$.
Thus the right-hand side is non-vanishing by our assumptions,
and this contradiction proves the claim.
\end{proof}

Lemma \ref{appthm19} will be used to prove the following result for homogeneous systems on the cotangent bundle.
First, recall that if $M$ is the map given by \eqref{def:radialmultiplier}, then
the radial vector field $\rho\in T(T^\ast(X)\smallsetminus 0)$ is defined by
\[
\rho f = \frac{d}{dt} M_t^\ast f|_{t=1}, \quad f\in C^\infty(T^\ast(X)\smallsetminus 0).
\]
In terms of local coordinates
we have $\rho(w)=\xi\partial_\xi$ if $w=(x,\xi)$, see the discussion
following~\cite[Definition~21.1.8]{ho3}.
Moreover, if $f$ is homogeneous of degree $\ell$, then differentiation gives $\rho f=\ell f$
by Euler's homogeneity relation.

\begin{prop}\label{appthm26}
For $j\ge 1$ let $p,q_j,g_j\in C^\infty(T^\ast (\mathbb{R}^n)\smallsetminus 0)$
be $N\times N$ systems, where
$p$ and $q_j$ are homogeneous of degree $m$ and $g_j$ is homogeneous of degree $0$.
Let $\{\gamma_j\}_{j=1}^\infty$ be a sequence in $T^\ast (\mathbb{R}^n)\smallsetminus 0$
such that $\gamma_j\to\gamma$ as $j\to\infty$, and assume that $p(\gamma)=p(\gamma_j)=0$ for all $j$.
Assume also that $p$ is of principal type at $\gamma$, that is, there exists a
tangent vector $\partial_\nu\in T_\gamma T^\ast(\mathbb{R}^n)$ such that
\[
\partial_\nu p(\gamma): \kernel p(\gamma)\longrightarrow \cokernel p(\gamma)=\C^N/\ran p(\gamma)
\]
is bijective.
If there exists an $N\times N$ system $q\in C^\infty(T^\ast(\mathbb{R})^n\smallsetminus 0)$,
homogeneous of degree $m$, such that
\[
\partial_x^\alpha\partial_\xi^\beta q(\gamma)
=\lim_{j\to\infty} \partial_x^\alpha\partial_\xi^\beta q_j(\gamma_j)
\]
for all $(\alpha,\beta)\in\mathbb{N}^n\times\mathbb{N}^n$,
and if $q_j-pg_j$ vanishes of infinite order at $\gamma_j$ for all $j$, then there
exists an $N\times N$ system $g\in C^\infty(T^\ast(\mathbb{R}^n)\smallsetminus 0)$, homogeneous
of degree $0$, such that $q-pg$ vanishes of infinite order at $\gamma$.
Furthermore,
\begin{equation}\label{eq:app27}
\partial_x^\alpha\partial_\xi^\beta g(\gamma)
=\lim_{j\to\infty} \partial_x^\alpha\partial_\xi^\beta g_j(\gamma_j)
\end{equation}
for all $(\alpha,\beta)\in\mathbb{N}^n\times\mathbb{N}^n$.
\end{prop}

\begin{proof}
Let $\pi : T^\ast (\mathbb{R}^n)\smallsetminus 0 \rightarrow S^\ast (\mathbb{R}^n)$ be the projection,
and identify $S^\ast(\R^n)$ with $\R^n\times S^{n-1}$.
We have $\kernel p(\gamma)=\C^N
=\cokernel p(\gamma)$, so
$\partial_\nu p(\gamma)$ is invertible.
It follows that
$\partial_{\lambda\nu} p(\gamma)=\lambda\partial_{\nu} p(\gamma)$ is invertible for all $\lambda>0$.
With $\gamma=(x_0,\xi_0)$ and $\nu=(\nu_1,\nu_2)$ this implies that
$\partial_{\mu} p(\pi (\gamma))$ is invertible for $\mu=(\nu_1,\nu_2/|\xi_0|)$ since
\begin{align*}
\partial_\nu p(\gamma)&=\frac{d}{dt}p(\gamma+t\nu)|_{t=0}\\
&=\frac{d}{dt} ( |\xi_0|^m p ( x_0+t\nu_1,(\xi_0+t\nu_2)/|\xi_0| ) )|_{t=0}
=|\xi_0|^m\partial_\mu p(\pi (\gamma)).
\end{align*}
By using the homogeneity of $p$, $q$, $q_j$ and $g_j$ we may then
assume that $\gamma$ and $\gamma_j$ belong to $S^\ast (\mathbb{R}^n)$
for $j\ge 1$ to begin with, and that $\partial_{\nu} p(\gamma)$ is
invertible with $\nu$ replaced by $\mu$.

We may also assume that $\nu$ is a tangent
vector $\nu\in T_\gamma S^\ast (\mathbb{R}^n)$.
Indeed, the radial vector field $\rho$
applied $k$ times to $a\in C^\infty (T^\ast(\mathbb{R}^n)\smallsetminus 0)$
equals $\ell^ka$ if $a$ is homogeneous of degree $\ell$. 
For any point $w\in S^\ast(\mathbb{R}^n)$
with $w=(w_x,w_\xi)$ in local coordinates on $T^\ast(\mathbb{R}^n)$
it is easy to see that
\[
T_w S^\ast (\mathbb{R}^n) =\{ (u,v)\in \mathbb{R}^n\times \mathbb{R}^n :
\langle w_\xi , v \rangle = 0\}.
\]
Therefore a basis for $T_w S^\ast (\mathbb{R}^n)$ together with
the radial vector field $\rho(w)$ at $w$ constitutes
a basis for $T_w T^\ast (\mathbb{R}^n)$. By these considerations it follows that
$\partial_\nu$ cannot be a multiple of the radial vector field
at $\gamma$ since $\partial_\nu p(\gamma)$ is invertible while
$p(\gamma)=0$. Hence, $\partial_\nu=c\rho(\gamma)+\partial_{\tilde{\nu}}$
for some $c\in\R$ and $0\ne \tilde{\nu}\in T_\gamma S^\ast(\R^n)$.
Again, since $p(\gamma)=0$ we have $\partial_{\tilde \nu}p(\gamma)=\partial_\nu p(\gamma)$
by Euler's homogeneity relation, which proves the claim.
Note that these arguments also show that if
we can find a homogeneous matrix valued function $g$ such
that $q-pg$ vanishes of infinite order in the directions
$T_\gamma S^\ast(\mathbb{R}^n)$, then $q-pg$ vanishes of infinite order
at $\gamma$, for the derivatives involving the radial direction are
determined by lower order derivatives in the directions
$T_\gamma S^\ast(\mathbb{R}^n)$.

Write $p(x,\xi)=|\xi|^m\pi^\ast p_s(x,\xi)$, where $p_s=p\circ\pi$ is the
restriction of $p$ to $S^\ast(\R^n)$. Doing the same for $q$, $q_j$ and $g_j$
we find by the hypotheses of the proposition
together with an application of Lemma \ref{appthm19},
that there exists a matrix valued function $g_s\in C^\infty (S^\ast (\mathbb{R}^n),\mathcal L_N)$,
such that $q_s-p_sg_s$ vanishes of infinite order at $\gamma$ and
\eqref{eq:app20} holds for $g_s$, interpreted in the appropriate
sense for a local frame for $S^\ast(\R^n)$.
The function $g(x,\xi)=\pi^\ast g_s(x,\xi)$
is homogeneous of degree $0$ and
coincides with $g_s$ on $S^\ast (\mathbb{R}^n)$. In particular,
all derivatives of $g$ and $g_s$ in the directions $T_\gamma S^\ast(\mathbb{R}^n)$
are equal at $\gamma$.
Thus, by the arguments above
we conclude that $q-pg$ vanishes of infinite order at $\gamma$.
Since $g$ and $g_j$ are homogeneous of degree $0$, the same
arguments also imply that \eqref{eq:app27} holds for $g$,
which completes the proof.
\end{proof}

Using Lemma \ref{appthm19elliptic}
in place of Lemma \ref{appthm19}
we obtain the following result for elliptic systems corresponding to Proposition \ref{appthm26}.

\begin{prop}\label{appthm26elliptic}
For $j\ge 1$ let $p,q_j,g_j\in C^\infty(T^\ast (\mathbb{R}^n)\smallsetminus 0)$
be $N\times N$ systems, where
$p$ and $q_j$ are homogeneous of degree $m$ and $g_j$ is homogeneous of degree $0$.
Let $\{\gamma_j\}_{j=1}^\infty$ be a sequence in $T^\ast (\mathbb{R}^n)\smallsetminus 0$
such that $\gamma_j\to\gamma$ as $j\to\infty$, and assume that
$|p(\gamma)|$ and $|p(\gamma_j)|$ are non-vanishing for all $j$,
where $|p|$ is the determinant of $p$.
If there exists an $N\times N$ system $q\in C^\infty(T^\ast(\mathbb{R})^n\smallsetminus 0)$,
homogeneous of degree $m$, such that
\[
\partial_x^\alpha\partial_\xi^\beta q(\gamma)
=\lim_{j\to\infty} \partial_x^\alpha\partial_\xi^\beta q_j(\gamma_j)
\]
for all $(\alpha,\beta)\in\mathbb{N}^n\times\mathbb{N}^n$,
and if $q_j-pg_j$ vanishes of infinite order at $\gamma_j$ for all $j$, then there
exists an $N\times N$ system $g\in C^\infty(T^\ast(\mathbb{R}^n)\smallsetminus 0)$, homogeneous
of degree $0$, such that $q-pg$ vanishes of infinite order at $\gamma$.
Furthermore,
\begin{equation*}
\partial_x^\alpha\partial_\xi^\beta g(\gamma)
=\lim_{j\to\infty} \partial_x^\alpha\partial_\xi^\beta g_j(\gamma_j)
\end{equation*}
for all $(\alpha,\beta)\in\mathbb{N}^n\times\mathbb{N}^n$.
\end{prop}

\begin{proof}
Let $\pi:T^\ast(\R^n)\to S^\ast(\R^n)$ be the projection,
and identify $S^\ast(\R^n)$ with $\R^n\times S^{n-1}$. Arguing as
in the proof of Proposition \ref{appthm26}, it follows
by homogeneity that all assumptions continue to hold after projecting
onto the cosphere bundle. An application of Lemma \ref{appthm19elliptic}
yields the existence of a matrix valued function $g_s\in C^\infty(S^\ast(\R^n),\mathcal L_N)$
for which the pullback $g=\pi^\ast g_s$ has the required properties. This completes the proof.
\end{proof}

\end{document}